\documentclass[a4paper]{amsart}
\usepackage{graphicx}
\usepackage{pslatex}
\usepackage{amsmath,amsfonts}
\usepackage{amstext}
\usepackage{amsthm}	
\usepackage{amssymb}
\usepackage{rotating}
\usepackage{multicol}
\usepackage{verbatim}
\usepackage{rotating}
\usepackage{setspace}
\usepackage{mathrsfs}
\usepackage{dsfont}
\usepackage{epstopdf}
\usepackage[dvipsnames]{xcolor}
\usepackage{fancyhdr}
\usepackage{textcomp}
\usepackage{indentfirst}
\usepackage{latexsym,verbatim,pdfsync,color}

\newtheorem{thm}{Theorem}[section]
\newtheorem{cor}[thm]{Corollary}
\newtheorem{prop}[thm]{Proposition}
\newtheorem{lem}[thm]{Lemma}
\newtheorem*{thm*}{Theorem}

\theoremstyle{definition}
\newtheorem{defn}[thm]{Definition}
\newtheorem{oss}[thm]{Observation}
\newtheorem{ese}[thm]{Example}
\newtheorem{rmk}[thm]{Remark}

\newtheorem*{nota}{Notation}
\newtheorem*{guess}{Guess}
\newtheorem*{rmk*}{Remark}
\newcommand{\bequa}{\begin{equation}}
\newcommand{\eequa}{\end{equation}}
\newcommand{\bthm}{\begin{thm}}
\newcommand{\bproof}{\begin{proof}}
\newcommand{\eproof}{\end{proof}}
\newcommand{\ethm}{\end{thm}}
\newcommand{\bdefn}{\begin{defn}}
\newcommand{\edefn}{\end{defn}}
\newcommand{\bprop}{\begin{prop}}
\newcommand{\eprop}{\end{prop}}
\newcommand{\bcor}{\begin{cor}}
\newcommand{\ecor}{\end{cor}}                   
\newcommand{\blem}{\begin{lem}}
\newcommand{\elem}{\end{lem}}
\newcommand{\bese}{\begin{ese}}
\newcommand{\eese}{\end{ese}}
\newcommand{\boss}{\begin{oss}}                 
\newcommand{\eoss}{\end{oss}}
\newcommand{\bnota}{\begin{nota}}
\newcommand{\enota}{\end{nota}}
\newcommand{\bguess}{\begin{guess}}
\newcommand{\eguess}{\end{guess}}
\newcommand{\brmk}{\begin{rmk}}
\newcommand{\ermk}{\end{rmk}}

\newcommand{\dsR}{\mathds{R}}
\newcommand{\dsZ}{\mathds{Z}}
\newcommand{\dsC}{\mathds{C}}
\newcommand{\dsT}{\mathds{T}}

\newcommand{\clC}{\mathcal{C}}

\newcommand{\clF}{\mathcal{F}}
\newcommand{\clH}{\mathcal{H}}
\newcommand{\clL}{\mathcal{L}}

\newcommand{\clW}{\mathcal{W}}

\newcommand\norm[2]{{\Vert{#1}\Vert_{#2}}}

\newcommand\quant{\advance\quantno by1
                     
\ifnum\quantno=1\qquad\else\quad\fi\forall }

\newcommand\opnorm[2]{|\!|\!| {#1} |\!|\!|_{#2}}

\DeclareMathOperator{\ch}{Ch}
\DeclareMathOperator{\sh}{Sh}

\newcommand{\lt}{\left(}
\newcommand{\rt}{\right)}

\newcommand{\lgra}{\left\{}
\newcommand{\rgra}{\right\}}

\def\im{\operatorname{Im}}
\def\re{\operatorname{Re}}

\newcommand{\pih}{\frac{\pi}{2}}

\def\p{\partial }
\def\II{I\!I}

 \title[Hardy Spaces and the Szeg\H{o} projection of the non-smooth worm domain $D'_\beta$]{Hardy Spaces and the Szeg\H{o} projection \\of the non-smooth worm domain $D'_\beta$}

 \author{Alessandro Monguzzi}
 \date{\today}

\numberwithin{equation}{section}
\begin{document}

\begin{abstract}
We define Hardy spaces $H^p(D'_\beta)$, $p\in(1,\infty)$, on the non-smooth worm domain $D'_\beta=\{(z_1,z_2)\in\dsC^2:|\im z_1-\log |z_2|^2|<\pih, |\log |z_2|^2|<\beta-\pih\}$ and we prove a series of related results such as the existence of boundary values on the distinguished boundary $\p D'_\beta$ of the domain and a Fatou-type theorem (i.e., pointwise convergence to the boundary values). Thus, we study the Szeg\H{o} projection operator $\widetilde{S}$ and the associated Szeg\H{o} kernel $K_{D'_\beta}$. More precisely, if $H^p(\partial D'_\beta)$ denotes the space of functions which are boundary values for functions in $H^p(D'_\beta)$,  we prove  that the operator $\widetilde{S}$ extends to a bounded linear operator 
$$
\widetilde{S}: L^p(\p D'_\beta)\to H^p(\p D'_\beta)
$$
for every $p\in(1,+\infty)$ and
$$
\widetilde{S}: W^{k,p}(\p D'_\beta)\to W^{k,p}(\p D'_\beta)
$$ 
for every $k>0$. Here $W^{k,p}$ denotes the Sobolev space of order $k$ and underlying $L^p$ norm, $p\in(1,\infty)$. As a consequence of the $L^p$ boundedness of $\widetilde{S}$, we prove that $H^p(D'_\beta)\cap\clC(\overline{D'_\beta})$ is a dense subspace of $H^p(D'_\beta)$.

\end{abstract}

\address{Dipartimento di Matematica, Universit\`{a} degli Studi di Milano, Via C. Saldini 50, 20133 Milano, Italy}
\email{alessandro.monguzzi@unimi.it}

\thanks{The author partially supported by the grant PRIN 2010-11 {\em Real and Complex Manifolds: Geometry, Topology and Harmonic Analysis} of the Italian Ministry of Education (MIUR)}

%
%
\maketitle
\section{Introduction and Main Results}

Given a domain $\Omega\subseteq\dsC^n$, it is a classical problem to study the Hardy spaces of holomorphic functions and the Szeg\H{o} projection operator associated to this domain. If $\rho$ is a defining function for $\Omega$, i.e., $\Omega=\{z\in\dsC^n:\rho(z)<0\}$ and $|\nabla\rho|\neq 0$ on the boundary of $\Omega$, a standard way to define the Hardy spaces $H^{p}(\Omega)$, $p\in(1,\infty)$, is to consider a family of approximating subdomains $\{\Omega_\varepsilon\}_{\varepsilon>0}$ where $\Omega_\varepsilon=\{z\in\dsC^n: \rho(z)<-\varepsilon\}$.  Then
$$
H^p(\Omega):= \left\{F\ \text{holomorphic in}\ \Omega: \|F\|^p_{H^p(\Omega)}=\sup_{\varepsilon>0}\int_{b\Omega_\varepsilon}|F(\zeta)|^p\ d\sigma_\varepsilon<\infty\right\},
$$
where $b\Omega_\varepsilon$ is the topological boundary of $\Omega_\varepsilon$ and $d\sigma_\varepsilon$ is the euclidean measure induced on $b\Omega_\varepsilon$ .

Every function $F$ in $H^p(\Omega)$ admits a boundary value function $\widetilde{F}$ and the linear space of these boundary value functions defines a closed subspace of $L^p(b\Omega)$ which we denote by $H^p(b\Omega)$. In the special case $p=2$, the orthogonal projection 
$$
S_{\Omega}:L^2(b\Omega)\to H^2(b\Omega)
$$
is called the Szeg\H{o} projection operator associated to $\Omega$ and it has an integral representation by means of an integral kernel known as Szeg\H{o} kernel. We refer to \cite{MR0473215} for more details. 

The geometry of the domain $\Omega$ affects the regularity of $S_\Omega$ and this problem has been extensively studied in the last 40 years. There is a number of results regarding the regularity of the Szeg\H{o} projection in Sobolev scale for many classes of domains: strictly pseudoconvex domains \cite{MR0450623}, smooth bounded complete Reinhardt domains in $\dsC^n$ \cite{MR773403, MR835396}, domains satisfying Catlin's property  $P$ \cite{MR871667}, complete Hartogs domains in $\dsC^2$ \cite{MR971689, MR999739}, domains of finite type in $\dsC^2$ \cite{MR979602}, domains that admit a defining function that is plurisubharmonic on the boundary \cite{MR1133741} and convex domains of finite type in $\dsC^n$ \cite{MR1452048}. We refer also to \cite{MR1094488} and \cite{MR1381988} for some results regarding the behavior of the Szeg\H{o} projection with respect to the real analyticity of functions.

We also have some results concerning the $L^p$ regularity of the Szeg\H{o} projection; in \cite{MR906810} the problem is studied for a particular family of weakly pseudoconvex domains, in \cite{MR1452048} the case of convex domains is treated , while in \cite{MR2030575} the authors deal with non-smooth, simply connected domains in the plane $\dsC$. More recently, Lanzani and Stein announced  in \cite{MR3084008} some  new results about the $L^p$ regularity of the Szeg\H{o} projection. They still deal with strictly pseudoconvex domains, but assume only $\clC^2$ boundary regularity. We also cite \cite{MR3145917} where a new definition of the Szeg\H{o} kernel is suggested.

The smooth worm domain $\clW=\clW_\beta$ does not
belong to any of the known situations. The domain $\clW$ was first introduced by Diederich and Forn\ae ss in \cite{MR0430315} as a counterexample to certain classical conjectures about the geometry of pseudoconvex domains. For $\beta>\pih$, the worm domain is defined by
\bequa\label{SmoothWorm}
\clW=\{(z_1,z_2)\in\dsC^2:|z_1-e^{i\log|z_2|^2}|^2<1-\eta(\log|z_2|^2), z_2\neq 0\},
\eequa
where $\eta$ is a smooth, even, convex, non-negative function on the real line, chosen so that $\eta^{-1}(0)=[-\beta+\pih,\beta-\pih]$ and so that $\clW$ is bounded, smooth and pseudoconvex. We refer to \cite{MR2393268} for a detailed history of the study of the worm domain $\clW$. Diederich and Forn\ae ss introduced this domain to provide an example of a smooth, bounded and pseudoconvex domain whose closure does not have a Stein neighborhood basis. Nearly 15 years after its introduction, the interest in the worm domain has been renewed since it turns out to be a counterexample to other important conjectures. Starting from ground-breaking works of Kiselman \cite{MR1128596} and Barrett \cite{MR1149863}, Christ \cite{MR1370592} finally proved that the Bergman projection   $P_{\clW}$ of the worm domain, i.e., the orthogonal projection of $L^2(\clW)$  onto the closed subspace of holomorphic functions, does not map $\clC^{\infty}(\overline{\clW})$ to $\clC^{\infty}(\overline{\clW})$. Therefore, the worm domain $\clW$ is a smooth bounded pseudoconvex domain which does not satisfy Bell's Condition $R$. This condition is closely related to the boundary regularity of biholomorphic mappings as has been shown in works of Bell \cite{MR625347} and Bell and Ligocka \cite{MR568937}. 
 Due to the results of Christ, the Bergman projection of the worm $\clW$ and other related domains has been extensively studied by many authors. We cite the recent papers \cite{MR2336324, MR2393268, MR2448387, MR2904008, 2014arXiv1408.0082B, 2014arXiv1410.8490K} and the references therein. We remark that the Szeg\H{o} projection can be considered a boundary analogue of the Bergman projection. Moreover, the regularity of the Szeg\H{o} projection, at least in a certain setting, has been proved in \cite{MR3272760} to be closely linked to the regularity of the complex Green operator in analogy with the Bergman projection and the $\overline{\partial}$-Neumann operator.

Due to the lack of general results concerning the regularity of the Szeg\H{o} projection of smooth bounded weakly pseudoconvex domains and the peculiar behavior of $P_\clW$, the study of the regularity of $S_\clW$ is an interesting starting point for research in this direction. The work presented here is a first step for this investigation. 

In \cite{MR1149863}, Barrett proves that the Bergman projection $P_{\clW}$ does not preserve Sobolev spaces of sufficiently high order with the aid two non-smooth model domains, namely,
$$
D'_\beta=\lgra (z_1,z_2)\in \dsC^2 : \left|\im z_1-\log|z_2|^2 \right|<\frac{\pi}{2}, \left|\log|z_2|^2\right|<\beta-\frac{\pi}{2}\rgra.
$$
and
$$
D_\beta=\left\{(\zeta_1,\zeta_2 )\in\dsC^2:\re(\zeta_1 e^{-i\log|\zeta_2|^2})>0, \big|\log|\zeta_2|^2\big|<\beta-\pih\right\}.
$$
We refer the reader to Figure \ref{WormStrip} for a representation of $D'_\beta$ in the $(\im z_1, \log|z_2|)$-plane.

Despite being biholomorphically equivalent, the domains $D'_\beta$ and $D_\beta$ have a fundamental difference. For each fixed $z_1\in\dsC$, the fiber in the second component of the domain $D'_\beta$, that is the set $\{z_2\in\dsC: (z_1,z_2)\in D'_\beta\}$, is connected. This is not the case for the domain $D_\beta$. 

The geometry of $D'_\beta$ allows to obtain precise information about its Bergman projection and this information can be transferred to the Bergman projection $P_{D_\beta}$ of $D_\beta$ by means of the transformation rule for the Bergman projection under biholomorphic mappings. Finally, Barrett uses an exhaustion argument to transfer the information from $P_{D_\beta}$ to $P_{\clW}$ and conclude the proof.

In analogy with the Bergman case, we want to obtain information on the Szeg\H{o} projection $S_{\clW}$ studying $S_{{D'_\beta}}$ and $S_{D_\beta}$, but new difficulties arise. Being $\clW$ a smooth bounded domain, there is no confusion about the definition of the projection $S_{\clW}$; a little more caution is required when considering the domain $D'_\beta$ and $D_\beta$. The Szeg\H{o} projection acts on functions defined on the boundary of the domain considered and in the case of $D'_\beta$ and $D_\beta$ we can choose to work with the topological or the distinguished boundary.  
Moreover, in general, we lack a transformation rule for the Szeg\H{o} projection under biholomorphic mappings, thus it is not immediate to transfer information from $S_{D'_\beta}$ to $S_{D_\beta}$. Lastly, Barrett's exhaustion argument does not apply to the Szeg\H{o} setting trivially.  
For these reasons, in this work we only focus on the domain $D'_\beta$. We postpone to a future paper the investigation of $D_\beta$ and $\clW$. 



Notice that the domain $D'_\beta$ is rotationally invariant in the $z_2$ variable and can be sliced in strips. More in detail, let us fix $z_2\in\dsC$ such that $|\log|z_2|^2|<\beta-\pih$; then, the set 
$$
D'_\beta (z_2)=\{z_1\in\dsC :(z_1,z_2)\in D'_\beta\}=\{z_1\in\dsC : |\im z_1-\log|z_2|^2|<\pih\}
$$ 
can be identified with a strip of width $\pi$.
All these characteristics will be reflected in our results. The rotationally invariance in the $z_2$-variable will allow us to use the theory of Fourier series, while the ``strip-like" geometry in the $z_1$-variable will make the results for the Hardy spaces on a strip available.

\begin{figure}[h]
\begin{center}
\includegraphics[width=12cm]{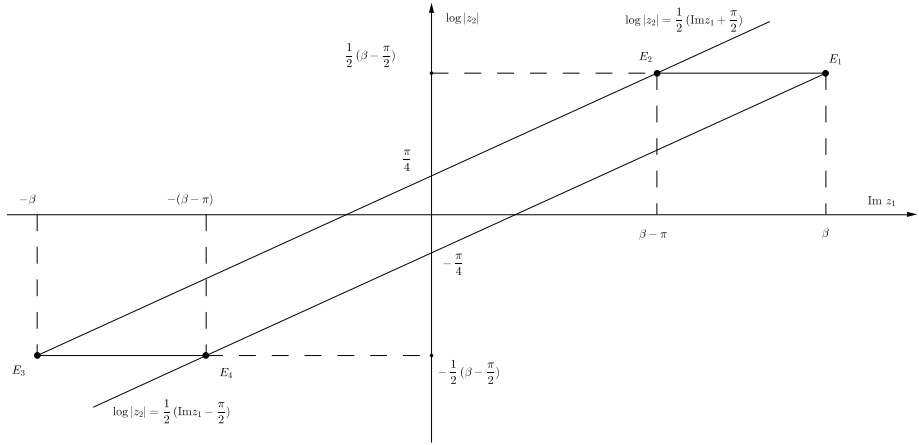}
\end{center}
\caption{A representation of the domain $D'_\beta$ in the $(\im z_1,\log|z_2|)$-plane.}\label{WormStrip}
\end{figure}

In order to define Hardy spaces on $D'_\beta$ we need to establish a $H^p$-type growth condition for holomorphic functions on $D'_\beta$. Instead of considering a growth condition on copies of the topological boundary $bD'_\beta$,  we decided to consider a growth condition on copies of the distinguished boundary $\p D'_{\beta}$. This seems to be a natural choice given the geometry of the domain. 

 In detail, the distinguished boundary $\partial D'_\beta$ is the set  
\begin{align}\label{BoundaryComponents}
\partial D'_\beta&=E_1\cup E_2\cup E_3\cup E_4,
\end{align}
where
\begin{align*}
&E_1=\lgra(z_1,z_2): \im z_1=\beta,\log|z_2|^2=\beta-\pih\rgra;\\
&E_2=\lgra(z_1,z_2): \im z_1=\beta-\pi,\log|z_2|^2=\beta-\pih\rgra\\
&E_3=\lgra(z_1,z_2)\: \im z_1=-\beta,\log|z_2|^2=-\lt\beta-\pih\rt\rgra;\\
&E_4=\lgra(z_1,z_2): \im z_1=-(\beta-\pi),\log|z_2|^2=-\lt\beta-\pih\rt\rgra.
\end{align*}

For every $p\in(1,\infty)$, we define the Hardy space $H^p(D'_\beta)$ as the function space
$$
H^p(D'_\beta)=\Big\{F\ \text{holomorphic in }\ D'_\beta: \|F\|^p_{H^p(D'_\beta)}=\sup_{(t,s)\in [0,\pih)\times[0,\beta-\pih)}\clL_pF(t,s)<\infty\Big\},
$$
where 
\begin{align}\label{WormGrowth}
\clL&_pF (t,s)=\\ \nonumber
&\int_{\dsR}\int_{0}^{1}\left|F\lt x+i(s+t),e^{\frac{s}{2}}e^{2\pi i\theta} \rt\right|^p d\theta dx +\int_{\dsR}\int_{0}^{1}\left|F\lt x-i(s+t),e^{-\frac{s}{2}}e^{2\pi i\theta} \rt\right|^p d\theta dx\\ \nonumber
&+\!\int_{\dsR}\int_{0}^{1}\!\left|F\lt x+i(s-t),e^{\frac{s}{2}}e^{2\pi i\theta} \rt\right|^p\! d\theta dx +\int_{\dsR}\int_{0}^{1}\!\left|F\lt x-i(s-t),e^{-\frac{s}{2}}e^{2\pi i\theta} \rt\right|^p\! d\theta dx. 
\end{align}

We emphasize that the domain $D'_\beta$ is not a product domain, while, on the other hand, every component $E_\ell$ of the distinguished boundary is and it can be identified with $\dsR\times\dsT$. 

The main results we obtain describe the good behavior in Sobolev and $L^p$ scale of the Szeg\H{o} projection $S_{D'_\beta}$ associated to the Hardy spaces just defined. A trivial remark is that, due to the definition of the spaces $H^p(D'_\beta)$, the associated Szeg\H{o} projection $S_{D'_\beta}$ acts on functions defined on the distinguished boundary of $D'_\beta$.

\begin{nota}
Before stating our results, we describe here some of the notation used in the paper. As we already mentioned, the distinguished boundary $\p D'_\beta$ has $4$ different components, thus when considering a function $\varphi:\p D'_\beta\to\dsC$ we actually mean a vector $\varphi=(\varphi_1, \varphi_2, \varphi_3, \varphi_4)$ where each function $\varphi_\ell$ is considered as defined on the component $E_\ell, \ell=1,\ldots,4$ of the distinguished boundary. Recall again that each $E_\ell$ can be identified with $\dsR\times\dsT$.

Given $p\in(1,\infty)$, the space $L^p(\p D'_\beta)$ is the function space
$$
L^p(\p D'_\beta)=\{\varphi=(\varphi_1,\varphi_2,\varphi_3,\varphi_4): |\!|\varphi|\!|^p_{L^p}=\sum_{\ell=1}^4|\!|\varphi_\ell|\!|^p_{L^p(\dsR\times\dsT)}<\infty\}.
$$
We use the notation $H^p(\p D'_\beta)$ to denote the closed subspace of $L^p(\p D'\beta)$ consisting of functions that are boundary values for functions in $H^p(D'_\beta)$. If $F\in H^p(D'_\beta)$, we use the notation $\widetilde F$ to denote the boundary value function of $F$. To be consistent with this convention, from now on the Szeg\H{o} projection associated to $D'_\beta$ will be denoted by $\widetilde{S}$, i.e., the operator $\widetilde S$ is the Hilbert space orthogonal projection 
$$
\widetilde{S}: L^2(\p D'_\beta)\to H^2(\p D'_\beta).
$$

If $\psi$ is a function in $C^{\infty}_0(\dsR\times\dsT)$, we denote with $\clF_\dsR \psi(\xi,\hat{j}\,)$ the Fourier transform of $\psi$ in the first variable and the $j$th Fourier coefficient in the second, i.e.,
$$
\clF_\dsR \psi(\xi,\hat{j}\,)=\frac{1}{2\pi}\int_\dsR\int_0^1 \psi(x,\gamma)e^{-ix\xi}e^{-2\pi ij\gamma}\ d\gamma dx.
$$
If $\mu$ is a function in $\clC^\infty_0 (\dsR)$, we denote either by $\widehat \mu$ or $\clF [\mu]$ its Fourier transform. The inverse Fourier transform will be denoted either by $\check \mu$ or $\clF^{-1}[\mu]$.

Given $p\in(1,\infty)$ and a real number $k>0$, the Sobolev space $W^{k,p}(\p D'_\beta)$ is the function space
\begin{equation}\label{SobolevSpace}
W^{k,p}(\p D'_\beta)=\{\varphi=(\varphi_1,\varphi_2,\varphi_3,\varphi_4) : |\!|\varphi|\!|^p_{W^{k,p}(\p D'_\beta)}=\sum_{\ell=1}^4 |\!|\varphi_\ell|\!|^p_{W^{k,p}(\dsR\times\dsT)}<\infty\},
\end{equation}
where
$$
|\!|\varphi_\ell|\!|^p_{W^{k,p}(\dsR\times\dsT)}=\int_{\dsR\times\dsT}\Big|\sum_{j\in\dsZ}e^{2\pi i j\gamma}\clF^{-1}_\dsR\big[[1+j^2+(\cdot)^2]^{\frac{k}{2}}\clF_\dsR\varphi_\ell(\cdot,\widehat{j})\big](x)\Big|^p\ dxd\gamma.
$$
\end{nota}
\noindent We adopt the non-standard notation $\ch(x)$ and $\sh(x)$ for the hyperbolic functions $\cosh(x)$ and $\sinh(x)$.

The main results we prove are the following.
\begin{thm}\label{t:Lp}
The Szeg\H{o} projection $\widetilde{S}$ extends to a linear bounded operator
\begin{align*}
\widetilde{S}:\ L^p(\partial D'_\beta)&\to H^p(\partial D'_\beta)\\
  \varphi&\mapsto \widetilde{S\varphi}
\end{align*}
for every $p\in(1,\infty)$.
\end{thm}
\begin{thm}\label{t:Sobolev}
The Szeg\H{o} projection $\widetilde{S}$ extends to a linear bounded operator
\begin{align*}
\widetilde{S}:\  W^{k,p}(\partial D'_\beta)&\to W^{k,p}(\partial D'_\beta)\\
    \varphi&\mapsto \widetilde{S\varphi}
\end{align*}
for every $p\in(1,\infty)$ and real number $k>0$.
\end{thm}

Besides these theorems, we carefully study the spaces $H^p(D'_\beta)$ proving a series of results such as a Fatou type theorem (i.e., pointwise convergence to the boundary values), a Paley--Wiener type theorem for the space $H^2(D'_\beta)$ and a nice decomposition for the spaces $H^p(D'_\beta)$.

The paper is organized in the following way. In section \ref{strip} we recall some results concerning the Hardy spaces on a symmetric strip. The boundedness results of the singular integrals
which arise in this context are consequence of the standard theory of
Calder\'{o}n-Zygmund convolution operators, but, to the best of the
author's knowledge, they do not appear explicitly in the
literature. Therefore, we  give some hints for the proofs since we perform some
computations which will be used in the sections that follow. In section \ref{L2} we study in detail the Hilbert space $H^2(D'_\beta)$. In Section \ref{Lp} we study the spaces $H^p(D'_\beta)$, $p\in(1,\infty)$, and we prove Theorem \ref{t:Lp} and Theorem \ref{t:Sobolev}. 

Unless specified, we will use standard and self-explanatory notation. If necessary, we will point out at the beginning of each section the notation conventions. 

We will denote by $C$, possibly with subscripts, a constant that may change from place to place.

\section{Hardy spaces for a symmetric strip}\label{strip}

In the introduction we mentioned that the non-smooth worm domain $D'_\beta$ can be sliced in strips.  This feature of $D'_\beta$ will be fundamental in the development of the Hardy spaces $H^p(D'_\beta)$ since it will allow us to use the theory of Hardy spaces on a strip. Hence, we recall here some results concerning the $H^p(S_\beta)$ spaces where $S_\beta$ is the symmetric strip
$$
S_\beta=\{x+iy\in\dsC : |y|<\beta\}.
$$
The results contained in this section are well-known. The boundedness results of the singular integrals
which arise in this context are consequence of the standard theory of Calder\'{o}n-Zygmund convolution operators. Some of these results are contained in \cite{MR2323495} and \cite{MR0369703},  nevertheless, for the reader's convenience, we include here some details. For full details, we refer also to \cite{MonThesis}.

For every  $p\in(1,\infty)$, the Hardy space $H^p(S_\beta)$ is the function space
$$
H^p(S_\beta)=\left\{f\ \text{holomorphic in}\ S_\beta:\|f\|_{H^p(S_\beta)}<\infty\right\},
$$
where
\begin{align}\label{NormStrip}
\|f\|^p_{H^p(S_\beta)}=\sup_{0\leq y<\beta}\left[\int_{\dsR}|f(x+iy)|^p\ dy+\int_{\dsR}|f(x-iy)|^p\ dy\right].
\end{align}

By Mean Value Theorem, it is immediate to prove that 
\begin{equation}\label{UniformCompact}
\sup_{z\in K}|f(z)|\leq C_{K}\|f\|_{H^p(S_\beta)}
\end{equation}
where $K$ is a compact subset of $S_\beta$. 

Now, we recall the well-known Paley--Wiener Theorem for a strip, which relates the growth of a holomorphic function in a strip with the growth of the Fourier transform of its restriction to the real line. We refer to \cite{MR1451142} for the proof.
\begin{thm}\textbf{(Paley--Wiener Theorem for a strip)}\label{PWstrip} Let $f_0$ in $L^2(\dsR)$. Then the following are equivalent:
\begin{itemize}
\item[$(i)$] $f_0$ is the restriction to the real line of a function $F$ holomorphic in the strip $S_\beta$ such that
$$
\sup_{|y|<\beta}\int_{\dsR}|F(x+iy)|^2 \ dy<\infty;
$$
\item[$(ii)$]$e^{\beta|\xi|}\widehat{f_0}\in L^2(\dsR)$. 
\end{itemize}
Moreover, the following relationship holds
\begin{align}\label{PWstrip2}
F(z)&=\frac{1}{2\pi}\int_{\dsR}\widehat{f_0}(\xi)e^{ iz\xi}\ d\xi=\clF^{-1}[e^{-\im z(\cdot)}\widehat{f_0}](\re z).
\end{align}
\end{thm}

Since $\p S_\beta$ has two boundary components and each of these components can be identified with the real line, the notation $L^p(\p S_\beta)$ denotes the space of functions $\varphi=(\varphi_+, \varphi_-)$ such that
$$
|\!|\varphi|\!|^p_{L^p(\p S_\beta)}=\int_{\dsR}|\varphi_+(x)|^p\ dx+\int_{\dsR}|\varphi_-(x)|^p\ dx<\infty. 
$$
We use the notation $\varphi_{\pm}$ since we think of $\varphi_+$ as a function defined on the upper boundary of the strip $S_\beta$ and of $\varphi_-$ as a function defined on the lower boundary.

Theorem \ref{PWstrip} guarantees that the function $\widetilde{f}_{\kappa}(x+\kappa i\beta):=\mathcal{F}^{-1}[e^{-\kappa\beta (\cdot)}\widehat{f}_0](x)$ is well-defined for $\kappa\in\{+,-\}$, therefore we can endow $H^2(S_\beta)$ with the inner product
 $$
 \left<f,g\right>_{H^2(S_\beta)}:=\big<\widetilde{f},\widetilde{g}\big>_{L^2(\partial S_\beta)},
 $$
The space $H^2(S_\beta)$ is a reproducing kernel Hilbert space with respect to this inner product. Hence, from \eqref{UniformCompact} and the Paley--Wiener Theorem, we obtain the following result.
\begin{thm}\label{StripRKHS}
The reproducing kernel of the Hardy space $H^2(S_\beta)$ is the function 
$$
K_{S_\beta}(w,z)=\frac{1}{4\pi}\int_{\dsR}\frac{e^{i(w-\overline{z})\xi}}{\ch[2\beta\xi]}\ d\xi=\frac{1}{2\beta\ch[\frac{\pi}{4\beta}(w-\overline{z})]}.
$$

Moreover, for all $f\in H^2(S_\beta)$,
$$
\lim_{y\to\pm\beta^{}}f(\cdot+iy)=\widetilde{f}_{\pm}
$$
where the limit holds in $L^2(\dsR)$ and for almost every $x$ in $\dsR$.

\end{thm} 
The integration against the kernel $K_{S_\beta}$ induces an operator which can be continuously extended to $L^p(\partial S_\beta)$ for every $p\in(1,\infty)$.
\begin{thm}\label{Lpstriscia}
Let $\varphi=(\varphi_+,\varphi_-)$ be a function in $L^2(\p S_\beta)\cap L^p(\p S_\beta)$, $p\in(1,\infty)$ and consider the operator $\varphi\mapsto S\varphi$ where
$$
S\varphi(z):= \int_{\dsR}\varphi_+(x)K_{S_\beta}(z,x+i\beta)\ dx+\int_\dsR \varphi_-(x)K_{S_\beta}(z,x-i\beta)\ dx.
$$

Then, the operator $\varphi\mapsto S\varphi$ extends to a bounded linear operator $S:L^p(\partial S_\beta)\to H^p(S_\beta)$.
\end{thm}
\begin{proof}
For future reference, we observe that for a function $\varphi\in L^2(\p S_\beta)\cap L^p(\p S_\beta)$ it holds
\begin{equation}\label{SFstrip}
S\varphi(z)=\clF^{-1}\Big[\frac{e^{-(\im z+\beta)(\cdot)}\widehat{\varphi_+}}{2\ch[2\beta(\cdot)]}\Big](\re z)+\clF^{-1}\Big[\frac{e^{-(\im z-\beta)(\cdot)}\widehat{\varphi_-}}{2\ch[2\beta(\cdot)]}\Big](\re z).
\end{equation}
This formula is immediately deduced from Theorem \ref{StripRKHS} and Plancherel's theorem.
The $L^p$ boundedness of the operator $S$ easily follows from Mihlin's multipliers theorem  (see, e.g., \cite[Chapter 5]{MR2445437}). 
\end{proof}

We conclude the section studying the $L^p$ regularity of the Szeg\H{o} projection associated to the spaces $H^p(S_\beta)$, $p\in(1,\infty)$.

Given $\varphi=(\varphi_+, \varphi_-)$ in $L^p(\p S_\beta)$, define the function $\widetilde{S\varphi}=(\widetilde{S\varphi}_+,\widetilde{S\varphi}_-)$ by
\begin{align*}
  &\widetilde{S\varphi}_+(x+i\beta)=\clF^{-1}\Big[\frac{e^{-2\beta(\cdot)}\widehat{\varphi_+}}{2\ch[2\beta(\cdot)]}\Big](x)+\clF^{-1}\Big[\frac{\widehat{\varphi_-}}{2\ch[2\beta(\cdot)]}\Big](x)\\
  & \widetilde{S\varphi}_-(x+i\beta)=\clF^{-1}\Big[\frac{\widehat{\varphi_+}}{2\ch[2\beta(\cdot)]}\Big](x)+\clF^{-1}\Big[\frac{e^{2\beta(\cdot)}\widehat{\varphi_-}}{2\ch[2\beta(\cdot)]}\Big](x).
\end{align*}
Consider now the operator $\varphi\mapsto \widetilde{S\varphi}$ and define
$$
H^p(\p S_\beta)=\{\varphi=(\varphi_+,\varphi_-)\in L^p(\p S_\beta):\exists f\in H^p(S_\beta)\ \text{s.t.}\ \widetilde{f}_+=\varphi_+,\widetilde{f}_-=\varphi_-\}.
$$
Then, $H^p(\p S_\beta)$ is a closed subspace of $L^p(\p D'_\beta)$ and the following theorem holds. 
\begin{thm}\label{SzegoStriscia}
Let $\varphi$ be a function in $L^p(\p S_\beta)$, $p\in(1,\infty)$.  Then,
\begin{equation}\label{LimitsStrip}
\lim_{y\to\beta^{\mp}}S\varphi(\cdot+iy))=\widetilde{S\varphi}
\end{equation}
where the limits are in $L^p(\dsR)$ and pointwise almost everywhere in $\dsR$. Moreover, the operator $\varphi\mapsto \widetilde{S\varphi}$ extends to a bounded linear operator 
\begin{align*}
\widetilde{S}:L^p(\p S_\beta)&\to H^p(\p S_\beta)
\end{align*}
for every $p\in(1,\infty)$. 
%
\end{thm}
\begin{proof}
The boundedness $L^p(\p S_\beta)\to L^p(\p S_\beta)$ of the operator $\widetilde{S}$ is immediately obtained by means of Mihlin's multipliers theorem. In order to conclude the proof is enough to prove that \eqref{LimitsStrip} holds. We do not include the details of the proof in full generality, but we give the general idea in a simplified situation. Namely, we prove \eqref{LimitsStrip} for a function $\varphi=(\varphi_+,0)$ in $L^p(\p S_\beta)$ meaning that $\varphi_{-}\equiv 0$. Instead of computing the limit for $y\to\beta^-$, we compute the equivalent $\lim_{\varepsilon\to 0^+}S\varphi[\cdot+i(\beta-\varepsilon)]$, where, using Theorem \ref{StripRKHS} and \eqref{SFstrip},
\begin{align} \nonumber
  S\varphi[&x+i(\beta-\varepsilon)]=\frac{1}{2\beta}\int_{\dsR}\frac{\varphi_+(x-y)}{\ch[\frac{\pi}{4\beta}(y+i(2\beta-\varepsilon))]}\ dy=\frac{1}{2\beta i}\int_{\dsR}\frac{\varphi_+(x-y)}{\sh[\frac{\pi}{4\beta}(y-i\varepsilon)]}\ dy\\ \nonumber \label{SFyEps}
  &=\frac{1}{2\beta}\int_{\dsR}\!\varphi_+(x-y)\frac{\ch[\frac{\pi y}{4\beta}]\sin[\frac{\pi\varepsilon}{4\beta}]}{\sh^2[\frac{\pi y}{4\beta}]+\sin^2[\frac{\pi \varepsilon}{4\beta}]} dy-\frac{i}{2\beta}\int_{\dsR}\!\varphi_+(x-y)\frac{\sh[\frac{\pi y}{4\beta}]\cos[\frac{\pi\varepsilon}{4\beta}]}{\sh^2[\frac{\pi y}{4\beta}]+\sin^2[\frac{\pi \varepsilon}{4\beta}]} dy\\ 
&=[K_\varepsilon\ast \varphi_+](x)-i[\widetilde{K}_\varepsilon\ast \varphi_+](x).
\end{align} 
Thus, we can study the kernels $K_\varepsilon$ and $\widetilde{K}_\varepsilon$ separately. It is not hard to prove that the family of functions $\{\frac{K_\varepsilon}{2}\}$ is a summability kernel, while the operator associated to the kernel $\widetilde{K}_\varepsilon$ can be studied comparing it to the singular integral operator $T$ defined on Schwartz functions by
$$
Tg(x)=\lim_{\varepsilon\to 0^+}\int_{|\frac{\pi x}{4\beta}|>\varepsilon}\frac{g(x-y)}{\sh[\frac{\pi y}{4\beta}]}\ dy.
$$
The conclusion follows now by the classical theory of Calder\'{o}n- Zygmund singular integral operators.
\end{proof}

\begin{rmk}\label{ContinuityBoundary}
We point out that if $\varphi\in \clC^\infty_0(\partial S_\beta)$, i.e., $\varphi_+$ and $\varphi_-$ belong to $\clC^\infty_0(\dsR)$, then $S\varphi$ belongs to $H^p(S_\beta)\cap \clC(\overline{S_\beta})$. This fact easily follows by  Lebesgue's dominated convergence theorem and \eqref{SFstrip}.
\end{rmk}

\section{Hardy spaces on $D'_\beta$: the $L^2$-theory}\label{L2}
In this section we study in detail the Hardy space $H^2(D'_\beta)$. One of the main feature of $H^2(D'_\beta)$ is that it can be written as direct sum of orthogonal subspaces and each of these subspaces turns out to be isometric to a weighted $H^2$ space on a strip (Theorem \ref{H2description}).

We use the notation $(z_1,. z_2)$ to denote an inner point of $D'_\beta$, while we use the notation $(\zeta_1,\zeta_2)$ to denote a point of the distinguished boundary $\p D'_\beta$.

 Using only the definition of $H^p(D'_\beta)$, it is not hard to prove that every function $F$ in $H^p(D'_\beta)$, $p\in(1,\infty)$, admits a boundary value function
$\widetilde{F}=(\widetilde{F}_1,\widetilde{F}_2,\widetilde{F}_3,\widetilde{F}_4)$
in $L^p(\partial D'_\beta)$ at least in a weak-$\ast$ sense. 

We define a family of functions $F_{t,s}=( F_{1,t,s}, F_{2,t,s}, F_{3,t,s}, F_{4,t,s})$ in $L^p(\p D'_\beta)$ by restricting $F$ to copies of the distinguished boundary $\p D'_\beta$ inside the domain $D'_\beta$. Namely, given a function $F$  in $H^p(D'_\beta)$, $p\in(1,\infty)$, for every $(t,s)\in[0,\pih)\times[0,\beta-\pih)$,  
we define
\begin{align*}
&F_{1,t,s}(\zeta_1,\zeta_2):=F\big(\re\zeta_1+i\frac{s+t}{\beta}\im\zeta_1,e^{-\frac{1}{2}(\beta-\pih-s)}\zeta_2\big);\\
&F_{2,t,s}(\zeta_1,\zeta_2):=F\big(\re\zeta_1+i\frac{s-t}{\beta-\pi}\im\zeta_1,e^{-\frac{1}{2}(\beta-\pih-s)}\zeta_2\big);\\
&F_{3,t,s}(\zeta_1,\zeta_2):=F\big(\re\zeta_1+i\frac{s+t}{\beta}\im\zeta_1,e^{\frac{1}{2}(\beta-\pih+s)}\zeta_2\big);\\
&F_{4,t,s}(\zeta_1,\zeta_2):=F\big(\re\zeta_1+i\frac{s-t}{\beta-\pi}\im\zeta_1,e^{\frac{1}{2}(\beta-\pih+s)}\zeta_2\big).
\end{align*}

The following proposition is elementary.
\begin{prop}\label{Weak*Hp}
Let $F$ be a function in $H^p(D'_\beta),p\in(1,\infty)$. Then, the following facts hold:
\begin{itemize}
\item [$(i)$]
there exists a subsequence $F_{(t,s)_n}$ which admits a weak-$\ast$ limit $\widetilde{F}$ in $L^p(\partial D'_\beta)$;
\item[$(ii)$]
for every compact subset $K$ of $D'_\beta$, the estimate
$$
\sup_{(z_1,z_2)\in K}|F(z_1,z_2)|\leq C_{K}\|F\|^p_{H^p}
$$ 
holds.
\end{itemize}
\end{prop}
We now focus on the space $H^2(D'_\beta)$ and prove that it admits a nice decomposition which allows to describe explicitly its reproducing kernel. 

\begin{thm}\label{H2description}
The Hardy space $H^2(D'_\beta)$ admits an orthogonal decomposition
\begin{equation}\label{DirectSumH2}
H^2(D'_\beta)=\bigoplus_{j\in\dsZ}\clH^2_j,
\end{equation}
where each $\clH^2_j$ is the subspace of $H^2(D'_\beta)$ defined as
\begin{equation}\label{Subspaces}
\mathcal{H}^2_j=\{F\in H^2(D'_\beta) : F(z_1,e^{i\theta}z_2)=e^{ij\theta}F(z_1,z_2)\}.
\end{equation}
Moreover, each subspace $\clH^2_j$ is isometric to the Hardy space of the strip $H^2(S_\beta)$ equipped with a weighted norm depending on $j$.
\end{thm}
The proof of Theorem \ref{H2description} will follow from a series of results that we state and prove separately for the reader's convenience.

Using the rotationally invariance of $D'_\beta$ in the $z_2$-variable and the connectedness of the set $D'_\beta(z_1)=\{z_2\in\dsC (z_1,z_2)\in D'_\beta\}$ for every fixed $z_1$, it is not hard to obtain the following proposition. See also \cite{MR1149863}  or  \cite{MR2393268}.
\begin{prop}\label{Barrett}
Let $F\in H^2(D'_\beta)$. Then,
\begin{align}\label{PointiwiseBar}
  F(z_1,z_2)&=\sum_{j\in\dsZ}\int_0^1 F(z_1,e^{2\pi i\theta}z_2)e^{-2\pi ij\theta}\ d\theta=\sum_{j\in\dsZ}F_j(z_1,z_2)=\sum_{j\in\dsZ}f_j(z_1) z_2^j,
\end{align}
where the series converges pointwise for every $(z_1,z_2)\in D'_\beta$ and each $f_j$ belongs to the Hardy space $H^2(S_\beta)$.
\end{prop}
Since each function $f_j$ belongs to the Hardy space $H^2(S_\beta)$, all the results contained in the previous section are available. In particular, we know that each function $f_j$ admits a boundary value function $\widetilde{f}_j$ in $L^2(\partial S_\beta)$. 


By the Paley--Wiener Theorem for the strip, the $H^2(D'_\beta)$ norm of each function $F_j$ in the sum \eqref{PointiwiseBar} is easily computed. In order to be consistent with the notation of Theorem \ref{PWstrip}, we denote by $f_{j,0}$ the restriction of the function $f_j$ to the real line.
\bprop\label{isometry}
Let $F_j(z_1,z_2)=f_j(z_1)z_2^j$ be a function in $\clH_j^2$, $j\in\dsZ$. Then,
\begin{align}
\nonumber\|F_j\|^2_{H^2(D'_\beta)}&=\bigg[ e^{j(\beta-\pih)}\|f_j[\cdot+i(\beta-\pih)]\|^2_{H^2(S_{\pih})}+e^{-j(\beta-\pih)}\|f_j[\cdot-i(\beta-\pih)]\|^2_{H^2(S_{\pih})}\bigg]\\ \nonumber
&=\frac{2}{\pi}\int_{\dsR}|\widehat{f_{j,0}}(\xi|^2\ch(\pi\xi)\ch[(2\beta-\pi)(\xi-\frac{j}{2})]\ d\xi.
\end{align}
In particular, 
$$
\sup\limits_{(t,s)}\clL_2F_j(t,s)=\lim\limits_{(t,s)\rightarrow(\pih,\beta-\pih)}\clL_2F_j(t,s).
$$
\eprop

\begin{rmk}\label{RemarkStrip}
Notice that, for every $j$ fixed, the quantity 
\begin{equation}\label{NormWeightedStrip}
\|f_j\|^2_{H^2_j(S_\beta)}=\|\widetilde{f_j}\|^2_{L^2_j(\p S_\beta)}:= \frac{2}{\pi}\int_{\dsR}|\widehat{f_{j,0}}(\xi)|^2\ch[\pi\xi]\ch[(2\beta-\pi)(\xi-\frac{j}{2})]\ d\xi
\end{equation}
defines a norm on $H^2(S_\beta)$ equivalent to the standard one. In conclusion, the previous proposition shows that $F_j\mapsto \widetilde{f}_j$ is an isometry between $\clH^2_j$ and $L^2_j(\partial S_\beta)$. This proves the second part of Theorem \ref{H2description}.
\end{rmk}
\bprop\label{orthogonality}
Let be $F$ a function in $H^2(D'_\beta)$. Then 
\begin{align*}
\|F\|_{H^2(D'_\beta)}^2&=\sup_{(t,s)}\sum_{j\in\dsZ}\clL_2F_j(t,s)=\sum_{j\in\dsZ}\sup_{(t,s)}\clL_2F_j(t,s)=\sum_{j\in\dsZ}\|F_j\|^2_{H^2(D'_\beta)},
\end{align*}
where the supremum is taken over $(t,s)\in[0,\pih)\times[0,\beta-\pih)$.
\eprop
\begin{proof}
We already know that $\|F\|_{H^2(D'_\beta)}^2=\sup\limits_{(t,s)}\sum_{j\in\dsZ}\clL_2F_j(t,s) $; it trivially follows from the orthogonality of trigonometric monomials. We would like to prove that it is possible to switch the supremum with the sum, i.e.,
$$
\sup_{(t,s)}\sum_{j\in\dsZ}\clL_2F_j(t,s)=\sum_{j\in\dsZ}\sup_{(t,s)}\clL_2F_j(t,s).
$$
Since we know from Proposition \ref{isometry} that $\sup\limits_{(t,s)}\clL_2F_j(t,s)=\lim\limits_{(t,s)}\clL_2F_j(t,s)$, we can conclude the result using monotone convergence.
\end{proof}
\begin{rmk}
From Proposition \ref{Barrett} and Proposition \ref{orthogonality} it is easily deduced that the series \eqref{PointiwiseBar} converges not only pointwise, but also in norm. That is,
$$
\norm{F-\sum_{j=-N}^{N}F_j}{H^2(D'_\beta)}\to 0
$$
as $N$ tends to $+\infty$.
\end{rmk}

Finally, we are  able to prove that a function $F\in
H^2(D'_\beta)$ admits boundary values in $L^2(\partial
D'_\beta)$.
  Let $F_{t,s}$ be the function defined in Proposition \ref{Weak*Hp}. 
\begin{prop}\label{BoundaryValuesH2WormStrip}
Let $F(z_1,z_2)=\sum\limits_{j\in\dsZ}f_j(z_1)z_2^j$ be a function in $H^2(D'_\beta)$. For $(\zeta_1,\zeta_2)\in\partial D'_\beta$ define
$$
\widetilde{F}(\zeta_1,\zeta_2):=\sum_{j\in\dsZ}\widetilde{F_j}(\zeta_1,\zeta_2)=\sum_{j\in\dsZ}\widetilde{f}_j(\zeta_1)\zeta_2^j.
$$
Then $F_{t,s}\to \widetilde{F}$ in $L^2(\p D'_\beta)$  as $(t,s)\to(\pih,\beta-\pih )$ and $\| F\|_{H^2(D'_\beta)}= \|\widetilde F\|_{L^2(\p D'_\beta)}$.
\end{prop}
\begin{proof}
Proposition \ref{orthogonality} guarantees that $\widetilde{F}$ is well defined.
Since $F$ is in $H^2(D'_\beta)$, it holds
$$
\|\widetilde{F}-F_{t,s} \|_{L^2(\partial D'_\beta)}^2
=\sum_{j\in\dsZ}\|\widetilde{F}_j-(F_{j})_{t,s}\|_{L^2(\partial D'_\beta)}^2<\infty.
$$
Moreover, $\|\widetilde{F}_j-(F_{j})_{t,s}\|_{L^2(\p D'_\beta)}^2\to 0$ as $(t,s)\to(\pih,\beta-\pih)$. Notice that
\begin{align}\label{Dominated}
\|\widetilde F-F_{t,s}\|^2_{L^2}\leq \|\widetilde{F}_j\|^2_{L^2}+\sup_{(t,s)}\|(F_j)_{t,s}\|^2_{L^2}+2\sup_{(t,s)}\|(F_j)_{t,s}\|_{L^2}\|\widetilde{F}_j\|_{L^2}
\end{align}
where the suprema are taken for $(t,s)\in[0,\pih)\times[0,\beta-\pih)$.  From Proposition \ref{orthogonality} and H\"older's inequality we obtain that the right-hand side of \eqref{Dominated} is summable. Therefore, by Lebesgue's dominated convergence theorem, we can switch the sum and the limit obtaining
\begin{align*}
\lim_{(t,s)\to(\pih,\beta-\pih)}\|\widetilde{F}-F_{t,s}\|_{L^2(\partial D'_\beta)}^2&=\sum_{j\in\dsZ}\lim_{(t,s)\to(\pih,\beta-\pih)}\!\|\widetilde{F}_{j}-(F_{j})_{t,s}\|_{L^2(\partial D'_\beta)}^2=0.
\end{align*}
The conclusion follows.
\end{proof}
As in the case of the strip, we identify the inner product in $H^2(D'_\beta)$ as an $L^2$ inner product on the distinguished boundary. Namely, given $F,G$ in $H^2(D'_\beta)$, we define
\begin{align}\label{ScalarProduct}
\left<F,G\right>_{H^2(D'_\beta)}:=\big<\widetilde{F},\widetilde{G}\big>_{L^2(\partial D'_\beta)}=\sum_{k=1}^4 \int_{E_k}\widetilde{F}(\zeta_1,\zeta_2)\overline{\widetilde{G}(\zeta_1,\zeta_2)}\ d\zeta_1 d\zeta_2.
\end{align}
The decomposition \eqref{DirectSumH2} is an orthogonal decomposition with respect to this inner product and Theorem \ref{H2description} is finally proved.
\subsection{The Szeg\H{o} kernel and projection of $D'_\beta$}
Before investigating the reproducing kernel $K_{D'_\beta}$ of $H^2(D'_\beta)$, we investigate the reproducing kernels of the subspaces $\clH^2_j$. The particular structure of each $\clH^2_j$ and Proposition \ref{isometry} allow us to look for the kernels of the spaces $H^2_j(S_\beta)$, that is the Hardy spaces of the strip endowed with the weighted norm defined by \eqref{NormWeightedStrip}.
\bprop
The reproducing kernel of $H^2_j(S_\beta)$ is the function
$$
k_j(z_1,z_2)=\frac{1}{8\pi}\int_{\dsR}\frac{e^{i(z_1-\overline{z_2})\xi}}{\ch[\pi\xi]\ch[(2\beta-\pi)(\xi-\frac{j}{2})]}\
d\xi .
$$
\eprop
\begin{proof}
Given $z_2$ in $S_\beta$ and $f\in H^2_j(S_\beta)$, using the  definition of reproducing kernel, Remark \ref{RemarkStrip} and Theorem \ref{PWstrip}, we obtain
\begin{align*}
f(z_2)&=\left<f,k_j(\cdot,z_2)\right>_{H^2_j(S_\beta)}\\
&=\frac{2}{\pi}\int_{\dsR}\widehat{f_0}(\xi)\overline{\widehat{k_{j,0}}(\xi,z_2)}\ch(\pi\xi)\ch[(2\beta-\pi)(\xi-\frac{j}{2})]\ d\xi\\
&=\frac{1}{2\pi}\int_{\dsR}\widehat{f_0}(\xi)e^{iz_2\xi}\ d\xi,
\end{align*}
It follows
$$
\widehat{k_{j,0}}(\xi,z_2)=\frac{1}{4}\frac{e^{-
i\overline{z_2}\xi}}{\ch(\pi\xi)\ch[(2\beta-\pi)(\xi-\frac{j}{2})]}
$$
and, using the inverse Fourier transform, we finally obtain
$$
k_{j}(z_1,z_2)=\frac{1}{8\pi}\int_{\dsR}\frac{e^{i(z_1-\overline{z_2})\xi}}{\ch[\pi\xi]\ch[(2\beta-\pi)(\xi-\frac{j}{2})]}\
d\xi. 
$$
\end{proof}
The reproducing kernel of $H^2(D'_\beta)$ is then given by 
\begin{align}\nonumber
K_{D'_\beta}[(w_1,w_2),(z_1,z_2)]&=\sum_{j\in\dsZ}{w_2}^j \overline{z_2}^jk_{j}(w_1,z_2)\\  \label{KernelSum}
&=\sum\limits_{j\in\dsZ}\frac {{w_2}^j \overline{z_2}^j}{8\pi}\int_{\dsR}\frac{e^{i(w_1-\overline{z_1 })\xi}}{\ch[\pi\xi]\ch[(2\beta-\pi)(\xi-\frac{j}{2})]}\ d\xi, 
\end{align}
where the series converges in $H^2(D'_\beta)$ for every fixed $(z_1,z_2)$ in $D'_\beta$.

If we fix a compact subset $K$ in $D'_\beta$, we have a stronger convergence for the series which defines $K_{D'_\beta}$.
\begin{prop}\label{UniformConvergenceKernel}
Let us consider $K_{D'_\beta}(z,\zeta)=K_{D'_\beta}[(z_1,z_2),(\zeta_1,\zeta_2)]$ where $(\zeta_1,\zeta_2)\in\partial D'_\beta$ and $(z_1,z_2)$ varies in a compact set $K\subseteq D'_\beta$. Then,
$$
\sum_{j\in\dsZ}\sup_{(z,\zeta)\in K\times\partial D'_\beta} \big|k_j(z_1,\zeta_1)z_2^j\overline{\zeta}_2^j\big|<\infty
$$
\end{prop}
\begin{proof}
We prove the proposition supposing that $(\zeta_1,\zeta_2)$ is in the component $E_1=\{(z_1,z_2):\im z_1=\beta, \log|z_2|^2=\beta-\pih\}$ of $\p D'_\beta$. The general case will follow analogously. In order to estimate the size of $k_j$, suppose for the moment that $j<0$. Then,
\begin{align*}
  |k_j(z_1,\zeta_1)|=|k_j(z_1,x+i\beta)|&\leq \int_{\dsR}\frac{e^{-[\im z_1+\beta]\xi}}{\ch[\pi\xi]\ch[(2\beta-\pi)(\xi-\frac{j}{2})]}\ d\xi\\
  &= \bigg(
  \int_{-\infty}^{\frac{j}{2}}+\int_{\frac{j}{2}}^{0}+\int_{0}^{+\infty}
  \bigg) 
\frac{e^{-[\im z_1+\beta]\xi}}{\ch[\pi\xi]\ch[(2\beta-\pi)(\xi-\frac{j}{2})]}\ d\xi.
\end{align*}
It follows that
\begin{align*}
  &\int_{-\infty}^{\frac{j}{2}}\frac{e^{-[\im
      z_1+\beta]\xi}d\xi}{\ch[\pi\xi]\ch[(2\beta-\pi)(\xi-\frac{j}{2})]} 
\approx \int_{-\infty}^{\frac{j}{2}}\frac{e^{-[\im z_1+\beta]\xi} d\xi}{e^{-\pi\xi}e^{-(2\beta-\pi)(\xi-\frac{j}{2})}}\approx\frac{e^{-j(\beta-\pih)}e^{\frac{j}{2}(\beta-\im z_1)}}{\beta-\im z_1};\\
  &\int_{\frac{j}{2}}^{0}\!\frac{e^{-[\im z_1+\beta]\xi}d\xi}{\ch[\pi\xi]\ch[(2\beta-\pi)(\xi-\frac{j}{2})]} 
\!\approx\! \int_{0}^{\frac{j}{2}}\!\frac{e^{-[\im z_1+\beta]\xi}d\xi}{e^{-\pi\xi}e^{(2\beta-\pi)(\xi-\frac{j}{2})}}\!\approx\! e^{j(\beta-\pih)}\frac{e^{-\frac{j}{2}[\im z_1+3\beta-2\pi]}-1}{\im z_1+3\beta-2\pi};\\
&\int_0^{+\infty}\frac{e^{-[\im z_1+\beta]\xi}d\xi}{\ch[\pi\xi]\ch[(2\beta-\pi)(\xi-\frac{j}{2})]}
\approx 
\int_0^{+\infty}\frac{e^{-[\im z_1+\beta]\xi}}{e^{\pi\xi}e^{(2\beta-\pi)(\xi-\frac{j}{2})}}\ d\xi\approx \frac{e^{j(\beta-\pih)}}{\im z_1+3\beta}.
\end{align*}
Notice that all these estimates do not depend on $\re\zeta_1$ and the term $\frac{e^{-\frac{j}{2}[\im z_1+3\beta-2\pi]}-1}{\im z_1+3\beta-2\pi}$ is not singular when $\im z_1+3\beta-2\pi\to 0$.
Finally, 
\begin{align*}
  &\sum_{j<0}|z_2|^je^{\frac{j}{2}(\beta-\pih)}|k_j(z_1,x+i\beta)|\\
  &\lesssim\!\sum_{j<0}\!\bigg[\frac{e^{\frac{j}{2}[\log|z_2|^2\!+\!\pih-\im z_1]}}{\beta-\im z_1}\!\!+\!\frac{e^{\frac{j}{2}[\log|z_2|^2-\im z_1+\pih]}-e^{\frac{j}{2}[\log|z_2|^2+3\beta-\frac{3\pi}{2}]}}{\im z_1+3\beta-2\pi}+\frac{e^{\frac{j}{2}[3\beta-\frac{3}{2}\pi+\log|z_2|^2]}}{\im z_1+3\beta}\bigg]
\end{align*}
and it is immediate to see that we get a uniform bound for $(z_1,z_2)\in K$. Analogous computations prove the estimate for the sum over positive $j$'s.
\end{proof}

We prove now that the integration against the kernel $K_{D'_\beta}$ not only reproduces function in $H^2(D'_\beta)$, but actually produces functions in $H^2(D'_\beta)$. 
\bprop\label{SFH2}
Let $\varphi$ be a function in $L^2(\partial D'_\beta)$. Then, the function 
$$
S\varphi(z_1,z_2):=\big<\varphi,K_{D'_\beta}[(\cdot,\cdot),(z_1,z_2)]\big>_{L^2(\partial D'_\beta)}
$$
is in $H^2(D'_\beta)$. Moreover, 
$$
\|S\varphi\|_{H^2(D'_\beta)}\leq \|\varphi\|_{L^2(\partial D'_\beta)}.
$$
\eprop
\begin{proof}
It is sufficient to prove the theorem for a function in $L^2(\p D'_\beta)$ of the form $\varphi=(\varphi_1, 0, 0, 0)$. The results for a general function $\varphi$ will follow by linearity. Therefore, by Plancherel's theorem,
\begin{align*}
\|\varphi\|^2_{L^2(\partial D'_\beta)}=\frac{1}{2\pi}\sum_{j\in\dsZ}\int_{\dsR}|\clF_{\dsR}\varphi_1(\xi,\widehat{j})|^2\ d\xi.
\end{align*}
The holomorphicity of $S\varphi$ is obtained using Proposition \ref{UniformConvergenceKernel} and Morera's theorem. It remains to prove that $S\varphi$ satisfies the $H^2$ growth condition. By \eqref{KernelSum} we obtain that
\begin{align*}
S\varphi(u+iv,re^{2\pi i\gamma})&=\int_{\dsR}\int_0^1\varphi_1(x,\theta)\sum\limits_{j\in\dsZ}k_{j}(u+iv,x+i\beta)r^j e^{2\pi ij\gamma}e^{\frac{j}{2}(\beta-\pih)}e^{-2\pi ij\theta}\ d\theta dx\\
&=\frac{1}{4}\sum_{j\in\dsZ} r^{j}e^{\frac{j}{2}(\beta-\pih)}e^{2 \pi
  ij\gamma}\clF^{-1}_{\dsR}\left[\frac{e^{-(v+\beta)(\cdot)}\clF_{\dsR}\varphi_1(\cdot,\hat{j}\,)}{\ch[\pi\cdot]\ch[(2\beta-\pi)(\cdot-\frac{j}{2})]}\right]\!(u). 
\end{align*} 
Hence,
\begin{align}\nonumber
  \int\limits_{\dsR}\!\int\limits_0^1\big|S\varphi[u+i(s+t), e^{\frac{s}{2}}e^{2\pi
    i\gamma}]\big|^2 d\gamma du&\!=\!\sum_{j\in\dsZ}\int_{\dsR}\Big|\frac{e^{-(s+\beta-\pih)(\xi-\frac{j}{2})}e^{-(\pih+t)\xi}\clF_{\dsR}\varphi_1(\xi,\hat{j}\,)}{8\pi\ch[\pi\xi]\ch[(2\beta-\pi)(\xi-\frac{j}{2})]}\Big|^2 d\xi\\ \label{L2BoundsWormStrip}
  &\leq\frac{1}{8\pi}\sum_{j\in\dsZ}\int_\dsR\Big|\clF_\dsR \varphi_1(\xi,\hat{j}\,)\Big|^2\ d\xi.
\end{align}
By analogous computations, we can estimate the other three terms of the $H^2$ growth condition \eqref{WormGrowth} and conclude the result by taking the supremum over $(t,s)\in[0,\pih)\times[0,\beta-\pih)$.
\end{proof}
\begin{rmk}
We report for completeness the explicit expression of $S\varphi$ given a general initial data $\varphi=(\varphi_1, \varphi_2, \varphi_3, \varphi_4)$ in $L^2(\partial D'_\beta)$. The formula is obtained combining \eqref{ScalarProduct} and \eqref{KernelSum}. Let $(u+iv, r^{2\pi i\gamma})$ in $D'_\beta$, then
\begin{align}\nonumber
  S\varphi(u+iv, re^{2\pi i\gamma})&=\left<\varphi, K_{D'_\beta}[(\cdot,\cdot),(u+iv, re^{2\pi i\gamma})]\right>_{L^2(\p D'_\beta)}\\ \nonumber
  &=\frac{1}{4}\sum_{j\in\dsZ}
  r^{j}e^{\frac{j}{2}(\beta-\pih)}e^{2 \pi
    ij\gamma}\clF^{-1}_{\dsR}\left[\frac{e^{-(v+\beta)(\cdot)}\clF_{\dsR}\varphi_1(\cdot,\hat{j}\,)}{\ch[\pi\cdot]\ch[(2\beta-\pi)(\cdot-\frac{j}{2})]}\right]
  (u)\\ \nonumber
  &\ +  \frac{1}{4}\sum_{j\in\dsZ}r^j e^{\frac{j}{2}(\beta-\pih)}e^{2\pi ij\gamma}\clF_\dsR^{-1}\left[\frac{e^{-(v+\beta-\pi)(\cdot)}\clF_{\dsR}\varphi_2(\cdot,\hat{j}\,)}{\ch[\pi\cdot]\ch[(2\beta-\pi)(\cdot-\frac{j}{2})]}\right](u)\\ \nonumber
  &\ + \
  \frac{1}{4}\sum_{j\in\dsZ}r^je^{-\frac{j}{2}(\beta-\pih)}e^{2\pi
    ij\gamma}\clF_{\dsR}^{-1}\left[\frac{e^{-(v-\beta)(\cdot)}\clF_{\dsR}\varphi_3(\cdot,\hat{j}\,)}{\ch[\pi\cdot]\ch[(2\beta-\pi)(\cdot-\frac{j}{2})]}\right]
  (u)\\  \label{SFGeneral}
  &\ + 
\frac{1}{4}\sum_{j\in\dsZ}r^{j}e^{-\frac{j}{2}(\beta-\pih)}e^{2\pi ij\gamma}\clF^{-1}_{\dsR}\left[\frac{e^{-(v-\beta+\pi)(\cdot)}\clF_{\dsR}\varphi_4(\cdot,\hat{j}\,)}{\ch[\pi\cdot]\ch[(2\beta-\pi)(\cdot-\frac{j}{2})]}\right](u).
\end{align}
\end{rmk}
Since $S\varphi$ is a function in $H^2(D'_\beta)$, we know it admits a boundary value function $\widetilde{S\varphi}=((\widetilde{S\varphi})_1, (\widetilde{S\varphi})_2, (\widetilde{S\varphi})_3, (\widetilde{S\varphi})_4)$ in $L^2(\p D'_\beta)$. For the reader's convenience, we adopt the notation $(\widetilde{S\varphi})_\ell:=\widetilde{S\varphi}_\ell$, $\ell=1,\ldots,4$.

 We obtain an explicit formula for $\widetilde{S\varphi}_1$, that is the boundary value function of $S\varphi$ on the component $E_1$ of $\p D'_\beta$; the formulas for the other components $\widetilde{S\varphi}_\ell, \ell=2,3,4,$ can be analogously deduced from \eqref{BoundaryComponents} and \eqref{SFGeneral}. 
 
 Given $\varphi$ in $L^2(\partial D'_\beta)$, we define a function $\psi=(\psi_1, \psi_2, \psi_3, \psi_4)$ in $L^2(\p D'_\beta)$ where we set
\begin{align} \label{GeneralBoundaryValuesWormStrip} \nonumber
  \psi_1(x+i\beta, e^{\frac{1}{2}(\beta-\pih)}e^{2\pi i\gamma}):=\frac{1}{4}&\sum_{j\in\dsZ}e^{2\pi ij\gamma}\clF_\dsR^{-1}\left[\frac{e^{-(2\beta-\pi)(\cdot-\frac{j}{2})}e^{-\pi(\cdot)}\clF_{\dsR}\varphi_1(\cdot,\hat{j}\,)}{\ch[\pi\cdot]\ch[(2\beta-\pi)(\cdot-\frac{j}{2})]}\right](x)\\ \nonumber
  &+\frac{1}{4}\sum_{j\in\dsZ}e^{2\pi ij\gamma}\clF_{\dsR}^{-1}\left[\frac{e^{-(2\beta-\pi)(\cdot-\frac{j}{2})}\clF_\dsR \varphi_2(\cdot,\hat{j}\,)}{\ch[\pi\cdot]\ch[(2\beta-\pi)(\cdot-\frac{j}{2})]}\right](x)\\ \nonumber
  &+\frac{1}{4}\sum_{j\in\dsZ}e^{2\pi ij\gamma}\clF_{\dsR}^{-1}\left[\frac{\clF_\dsR \varphi_3(\cdot,\hat{j}\,)}{\ch[\pi\cdot]\ch[(2\beta-\pi)(\cdot-\frac{j}{2})]}\right](x)\\
  &+\frac{1}{4}\sum_{j\in\dsZ}e^{2\pi ij\gamma}\clF_{\dsR}^{-1}\left[\frac{e^{-\pi(\cdot)}\clF_\dsR \varphi_4(\cdot,\hat{j}\,)}{\ch[\pi\cdot]\ch[(2\beta-\pi)(\cdot-\frac{j}{2})]}\right](x);
  \end{align}

Using the notation of Proposition \ref{Weak*Hp}, we have the following convergence result. \
\begin{prop}\label{L2NormBoundaryValuesWormStrip}
Let $\varphi$ be a function in $L^2(\partial D'_\beta)$.
Then,
\begin{align*}
  &\lim_{(t,s)\to(\pih,\beta-\pih)}\|[S\varphi]_{t,s}-\psi\|^2_{L^2(\p D'_\beta)}=\lim_{(t,s)\to(\pih,\beta-\pih)}\sum_{\ell=1}^4\|{[S\varphi}]_{\ell,t,s}-\psi_\ell\|_{L^2(E_\ell)}^2=0
\end{align*}ly
In particular, $\widetilde{S\varphi}=\psi$.
\end{prop}
\begin{proof}
We only prove explicitly that 
$$
\|[S\varphi]_{1,t,s}-\psi_1\|^2_{L^2(E_1)}\to 0
$$  
for a simpler function $\varphi$ of the form $\varphi=(\varphi_1, 0, 0, 0)$ in $L^2(\p D'_\beta)$.  The complete proof for a general function $\varphi$ is obtained with similar arguments. From \eqref{SFGeneral} and \eqref{GeneralBoundaryValuesWormStrip}, we obtain
\begin{align*}
  \|[S\varphi]_{1,t,s}-\psi_1&\|^2_{L^2(E_1)}=\|S\varphi(\cdot+i(s+t),e^\frac{s}{2}e^{2\pi i(\cdot)})-\psi_1(\cdot+i\beta,e^{\frac{1}{2}(\beta-\pih)}e^{2\pi i(\cdot)})\|^2_{L^2(\dsR\times\dsT)}\\
  &=\frac{1}{8\pi}\sum_{j\in\dsZ}\int_{\dsR}\Big|\clF_{\dsR}\varphi_1(\xi,\widehat{j})\frac{e^{-(s+\beta-\pih)(\xi-\frac{j}{2})}e^{-(\pih+t)\xi}-e^{-(2\beta-\pi)\xi}e^{-\pi\xi}}{\ch[\pi\xi]\ch[(2\beta-\pi)(\xi-\frac{j}{2})]}\Big|^2\ d\xi\\
  &\leq \frac{1}{8\pi}\sum_{j\in\dsZ}\int_{\dsR}\Big|\clF_{\dsR}\varphi_1(\xi,\widehat{j})\Big|^2\ d\xi\\
  &<\infty.
\end{align*}
By Lebesgue's dominated convergence theorem, we can conclude. 
\end{proof}
Let us define
$$
H^2(\partial D'_\beta):=\{\varphi=(\varphi_1, \varphi_2, \varphi_3, \varphi_4)\in L^2(\partial D'_\beta): \exists F\in H^2(D'_\beta)\ s.t.\ \varphi=\widetilde{F}   \}.
$$

We conclude the section with a Paley--Wiener type result. 
\begin{thm}\label{PWWorm}\textbf{(Paley--Wiener Theorem for $D'_\beta$)}
Let $\varphi=(\varphi_1, \varphi_2, \varphi_3, \varphi_4)$ be a function in $L^2(\partial D'_\beta)$. Then, $\varphi$ is in $H^2(\partial D'_\beta)$ if and only if
there exists a sequence of functions $\{g_j\}_{j\in\dsZ}$ such that
$$
\sum_{j\in\dsZ}\int_\dsR |\widehat{g_j}(\xi)|^2\ch[\pi\xi]\ch[(2\beta-\pi)(\xi-\frac{j}{2})]\ d\xi<\infty
$$
and 
\begin{align*}
  &\varphi_{1}(x+i\beta,e^{\frac{1}{2}(\beta-\pih)}e^{2\pi i\gamma})=\sum_{j\in\dsZ}\varphi_{1,j}(x+i\beta)e^{\frac{j}{2}(\beta-\pih)}e^{2\pi ij\gamma};\\
  &\varphi_{2}[x+i(\beta-\pi),e^{\frac{1}{2}(\beta-\pih)}e^{2\pi i\gamma}]=\sum_{j\in\dsZ}\varphi_{2,j}[x+i(\beta-\pih)]e^{\frac{j}{2}(\beta-\pih)}e^{2\pi i\gamma};\\
  &\varphi_{3}(x-i\beta,e^{-\frac{1}{2}(\beta-\pih)}e^{2\pi i\gamma})=\sum_{j\in\dsZ}\varphi_{3,j}(x-i\beta)e^{-\frac{j}{2}(\beta-\pih)}e^{2\pi ij\gamma};\\
   &\varphi_{4}[x-i(\beta-\pi),e^{-\frac{1}{2}(\beta-\pih)}e^{2\pi i\gamma}]=\sum_{j\in\dsZ}\varphi_{4,j}[x-i(\beta-\pih)]e^{-\frac{j}{2}(\beta-\pih)}e^{2\pi i\gamma},
\end{align*}
where, for every $j\in\dsZ$,
\begin{align*}
  &\varphi_{1,j}[x+i\beta]=\clF^{-1}_{\dsR}\Big[e^{-\beta(\cdot)}g_j(\cdot)\Big](x);\qquad \varphi_{2,j}[x+i(\beta-\pi)]=\clF^{-1}_{\dsR}\Big[e^{-(\beta-\pi)(\cdot)}g_j(\cdot)\Big](x);\\
  &\varphi_{3,j}(x+i\beta)=\clF^{-1}_{\dsR}\Big[e^{\beta(\cdot)}g_j(\cdot)\Big](x);\qquad \varphi_{4,j}[x-i(\beta-\pi)]=\clF^{-1}_{\dsR}\Big[e^{(\beta-\pi)(\cdot)}g_j(\cdot)\Big](x).
\end{align*}
Moreover,
\begin{equation}
S\varphi(u+iv, r e^{2\pi ij\gamma})=\sum_{j\in\dsZ}r^j e^{2\pi ij\gamma}\clF^{-}_{\dsR}\big[e^{-v(\cdot)}g_j\big](u).
\end{equation}
\end{thm}
\begin{proof}
Suppose that $\varphi$ belongs to $H^2(\partial D'_\beta)$. Then, the conclusion follows from Theorem \ref{H2description}, \eqref{PointiwiseBar} and the Paley--Wiener Theorem for a strip. Conversely, let $\{g_j\}$ be a sequence which defines $\varphi=(\varphi_1,\varphi_2,\varphi_3,\varphi_4)$ as in the hypothesis. It follows that $S\varphi$ belongs to $H^2(D'_\beta)$ and the formula in Definition \ref{GeneralBoundaryValuesWormStrip} guarantees that $\widetilde{S\varphi}_1=\varphi_1$. Analogously it can be proved for $k=2,3,4$. The proof is complete. 
\end{proof}

\section{Hardy Spaces on $D'_\beta$: the $L^p$-theory}\label{Lp}

In this section we extend the results we have seen so far to the case $p\in(1,\infty)$. In detail, we prove Theorems \ref{t:Lp} and \ref{t:Sobolev}, we prove that for every $p\in(1,\infty)$ the space $H^p(D'_\beta)$ admits a decomposition analogous to \eqref{DirectSumH2} (Proposition \ref{SumNormConvergence}) and we prove a Fatou-type Theorem, that is we prove that an appropriate restriction of a function $F$ in $H^p(D'_\beta)$, $p\in(1,\infty)$, converges to its boundary value function $\widetilde{F}$ pointwise almost everywhere (Theorem \ref{SimplyPC}).

For $y$ and $s$ such that $(x+iy, e^\frac{s}{2} e^{2\pi i\gamma})\in D'_\beta$, that is $|s|\in (0,\beta-\pih)$ and $|y-s|\in (0,\pih)$, we consider a family of operators $\{S_{y,s}\}_{y,s}$, where 
 \begin{equation}\label{Sys}
 S_{y,s}\varphi(x,\gamma):= S\varphi (x+iy, e^{\frac{s}{2}}e^{2\pi i\gamma})
 \end{equation}
 and the right-hand side of \eqref{Sys} is defined by \eqref{SFGeneral}.
 
We observe that the operator $\varphi\mapsto\widetilde{S\varphi}$ defined in Proposition \ref{L2NormBoundaryValuesWormStrip} and the operators $\varphi\mapsto S_{y,s}\varphi$ are well-defined for functions $\varphi=(\varphi_1, \varphi_2, \varphi_3, \varphi_4)$ where each $\varphi_\ell, \ell=1,\ldots,4$ is of the form
\bequa\label{FiniteSum}
\varphi_\ell(x,\gamma)=\sum_{ |j|<N}\varphi_\ell(x,j)e^{2\pi ij\gamma}
\eequa
with $\varphi_\ell(\cdot,j)$ in $C^{\infty}_0(\dsR)$ for every $j$ and the sum is over a finite number of $j$'s. Moreover, the set of functions $\varphi$ of such a form is dense in $L^p(\p D'_\beta)$.

For future reference, we point out that for a function $\varphi$ in $L^p(\partial D'_\beta)$ of the form $\varphi=(\varphi_1, 0, 0, 0)$, formulas  $\eqref{SFGeneral}$ and \eqref{GeneralBoundaryValuesWormStrip} reduce to
\begin{align}
\label{SimplySFGeneral}
&S_{y,s}\varphi(x,\gamma)=\sum_{j\in\dsZ}e^{2\pi ij\gamma}\clF^{-1}_{\dsR}\left[\frac{e^{-(\beta-\pih+s)(\cdot-\frac{j}{2})}e^{-(\pih-s+y)(\cdot)}\clF_{\dsR}\varphi_1(\cdot,\hat{j}\,)}{4\ch[\pi\cdot]\ch[(2\beta-\pi)(\cdot-\frac{j}{2})]}\right](x);\\
\label{SimplySF1}
  &\widetilde{S\varphi}_1(x+i\beta,e^{\frac{1}{2}(\beta-\pih)}e^{2\pi i\gamma})=\sum_{j\in\dsZ}e^{2\pi ij\gamma}\clF^{-1}_{\dsR}\left[\frac{e^{-(2\beta-\pi)(\cdot-\frac{j}{2})}e^{-\pi(\cdot)}\clF_{\dsR}\varphi_1(\cdot,\hat{j}\,)}{4\ch[\pi\cdot]\ch[(2\beta-\pi)(\cdot-\frac{j}{2})]}\right](x).
\end{align} 
To obtain \eqref{SimplySF1} we used also Proposition \ref{L2NormBoundaryValuesWormStrip}.
\begin{prop}\label{TransferredLpBounds}
Let $\varphi=(\varphi_1, \varphi_2, \varphi_3, \varphi_4)$ be a function in $L^p(\partial D'_\beta)$. Then, for every $p\in(1,\infty)$, 
$$
\|S_{y,s}\varphi\|_{L^p(\dsR\times\dsT)}\leq C_p \|\varphi\|_{L^p(\partial D'_\beta)},
$$
where the constant $C_p$ does not depend on $y$ and $s$.
\end{prop}
\begin{proof}
By density and linearity it suffices to prove the theorem for a function $\varphi$ of the form $\varphi= (\varphi_1, 0, 0, 0)$ where $\varphi_1(x,\gamma)=\sum\limits_{j=-N}^N \varphi_1(x,j)e^{2\pi ij\gamma}$ and each function $\varphi_1(\cdot,j)$ is in $C^\infty_0(\dsR)$. Define
$$
m_s(\xi-\frac{j}{2})=\frac{e^{-(\beta-\pih+s)(\xi-\frac{j}{2})}}{4\ch[(2\beta-\pi)(\xi-\frac{j}{2}]}\qquad\text{and}\qquad m'_{y,s}(\xi)= \frac{e^{-(\pih-s+y)\xi}}{\ch[\pi\xi]}.
$$
Then,
\begin{align}\label{SysComposition}
  S_{y,s}\varphi(x,\gamma)=[\lambda'_{y,s}\circ\lambda_s]\varphi(x,\gamma),
\end{align}
where 
\begin{align}
\lambda_s \varphi(x,\gamma)
 \label{lambdas}
&=\sum_{j=-N}^N e^{2\pi ij\gamma }\clF^{-1}_\dsR\left[m_s(\cdot-\frac{j}{2})\clF_\dsR \varphi_1(\cdot,j)\right](x)
\end{align}
and
\begin{align}
\lambda'_{y,s}\varphi(x,\gamma)
\label{lambdays}
&=\clF^{-1}_\dsR\left[m'_{y,s}(\cdot) \clF_\dsR \varphi_1(\cdot,\gamma)\right](x).
\end{align}
In fact, it is immediate to see from \eqref{lambdas} that
$$
\clF_\dsR\left[\lambda_s \varphi(\cdot,\gamma)\right](\xi)=\sum_{j=-N}^N e^{2\pi ij\gamma} m_s(\xi-\frac{j}{2})\clF_\dsR\varphi_1(\xi,j),
$$
hence from \eqref{lambdays} we obtain \eqref{SysComposition}. 
We recall that $y$ and $s$ are such that $(x+iy,e^{\frac{s}{2}}e^{2\pi i\gamma})$ is in $D'_\beta$, thus $|s|\in(0,\beta-\pih)$ and $|y-s|\in(0,\pih)$. Then, by Mihlin's multipliers theorem, it is not hard to prove that $m'_{y,s}$ is a multiplier of $L^p(\dsR)$ for every $p\in(1,\infty)$ with norm independent of $y$ and $s$. Thus the operator $\lambda'_{y,s}$ extends to a bounded linear operator $L^p(\dsR\times\dsT)\to L^p(\dsR\times\dsT)$ for every $p\in(1,\infty)$. About $\lambda_s$, by elementary properties of the Fourier transform, we have
\begin{align*}
  \lambda_s \varphi(x,\gamma)&=\frac{1}{2\pi}\int_\dsR \sum_{j=-N}^N e^{2\pi ij(\gamma+\frac{x}{4\pi})}m_s(\xi)\clF_\dsR \varphi_1(\xi+\frac{j}{2},j) e^{ix\xi}\ d\xi\\
  &=\frac{1}{2\pi}\int_\dsR \sum_{j=-N}^N e^{2\pi i j(\gamma+\frac{x}{4\pi})}m_s(\xi) \clF_\dsR[e^{-i\frac{j}{2}(\cdot)}\varphi_1(\cdot,j)](\xi)e^{ix\xi}\ d\xi.
\end{align*}
Therefore, by a change of variables and the periodicity of the exponential function,
\begin{align*}
  \int_{\dsR\times\dsT}|\lambda_s &\varphi(x,\gamma)|^p dxd\gamma\\
  &=\int_0^1\int_\dsR \bigg|\frac{1}{2\pi}\int_\dsR m_s(\xi)\sum_{j=-N}^N e^{2\pi ij\gamma}\clF_\dsR[e^{-i\frac{j}{2}(\cdot)}\varphi_1(\cdot,j)](\xi)e^{ix\xi}\ d\xi \bigg|^p dxd\gamma\\
  &=\int_0^1\int_\dsR \bigg|\frac{1}{2\pi}\int_\dsR m_s(\xi)\clF_\dsR\bigg[\sum_{j=-N}^N e^{-i\frac{j}{2}(\cdot)}\varphi_1(\cdot,j)e^{2\pi ij\gamma}\bigg](\xi)e^{ix\xi}\ d\xi\bigg|^p dxd\gamma.
\end{align*}
Again, by Mihlin's multipliers theorem, we obtain that $m_s$ is a multiplier of $L^p(\dsR)$ for every $p\in(1,\infty)$ with norm independent of $s$. Therefore, if we prove that the function $\sum\limits_{j=-N}^N e^{-i\frac{j}{2}t} \varphi_1(t,j)e^{2\pi ij\gamma}$ is in $L^p(\dsR\times\dsT)$, we will obtain the $L^p$ boundedness of the operator $\lambda_s$. By a change of variables and the periodicity of the exponential function, we have
\begin{align*}
  \int_\dsR\int_0^1 \!\bigg|\sum\limits_{j=-N}^N e^{-i\frac{j}{2}t} \varphi_1(t,j)e^{2\pi ij\gamma}\bigg|^p d\gamma dt&\!=\!\int_\dsR\int_0^1 \!\bigg|\!\sum\limits_{j=-N}^N \varphi_1(t,j)e^{2\pi ij\gamma}\bigg|^p d\gamma dt\!=\!\|\varphi\|^{p}_{L^p(\p D'_\beta)}.
\end{align*}
 Finally, we conclude the proof exploiting the boundedness of the operators $\lambda_s$ and $\lambda'_{y,s}$ and \eqref{SysComposition}.
\end{proof}

The last proposition allows us to prove that the operator $S$ defined by \eqref{SFGeneral} extends to a continuous operator with respect to the $L^p$ norm.
\begin{thm}\label{HolomorphicityHp}
For every $p\in(1,\infty)$, the operator $S$ extends to a bounded linear operator
$$
S:L^p(\partial D'_\beta)\to H^p( D'_\beta).
$$
\end{thm}
\begin{proof}
Suppose that $\varphi=(\varphi_1, 0, 0, 0) $ is a function in $L^p(\partial D'_\beta)\cap L^2(\partial D'_\beta)$. Then, Proposition \ref{SFH2} assures that $S\varphi$ is holomorphic on $D'_\beta$. Moreover, 
\begin{align}\nonumber
  \|&S\varphi\|^p_{H^p(D'_\beta)}=\sup_{(t,s)\in[0,\pih)\times[0,\beta-\pih)}\clL_pS\varphi(t,s)\\ \nonumber
   &=\sup_{(t,s)}\bigg[\|S_{s+t,s}\varphi\|^p_{L^p}+\|S_{s-t,s}\varphi\|^p_{L^p}+\|S_{-(s+t),-s}\varphi\|^p_{L^p}+\|S_{-(s-t),-s}\varphi\|^p_{L^p}\bigg]\\ \label{SFHp}
  &\leq C_p \|\varphi\|^p_{L^p(\p D'_\beta)}
\end{align}
with $C_p$ independent of $t$ and $s$ thanks to Proposition \ref{TransferredLpBounds}. Thus, we proved the theorem when $\varphi$ is in $L^p(\partial D'_\beta)\cap L^2(\partial D'_\beta)$. By density we obtain the proof for a general function $\varphi$ in $L^p(\p D'_\beta)$.
\end{proof}

It remains to prove that $S\varphi$ admits a boundary value function $\widetilde{S\varphi}$ in $L^p$. In order to keep the length of this work contained, we only prove explicitly that that the component $\widetilde{S\varphi}_1$ of $\widetilde{S\varphi}$ is the function defined by \eqref{GeneralBoundaryValuesWormStrip}.

\begin{thm}\label{LimitSF1}
Let $\varphi=(\varphi_1, \varphi_2, \varphi_3, \varphi_4)$ be a function in $L^p(\partial D'_\beta)$. Then, for every $p\in(1,\infty)$,
\begin{align}\label{Limit0}
  &\lim_{(t,s)\to(\pih,\beta-\pih)}\|S\varphi[\cdot+i(s+t),e^{\frac{s}{2}}e^{2\pi i\cdot}]-\widetilde{S\varphi}_1\|_{L^p(\dsR\times\dsT)}=0.
\end{align}
\end{thm}
The proof of the theorem will follow from a series of results that we state and prove separately. Let us fix some notation. Given $\varphi=(\varphi_1, 0, 0, 0)$, from \eqref{SimplySF1}, \eqref{SimplySFGeneral} and simple computations, we obtain
 \begin{align}\nonumber
  [&\widetilde{S\varphi}_1-S_{s+t,s}\varphi](x,\gamma)\\ \nonumber
  &=\sum_{j\in\dsZ}e^{2\pi ij\gamma}\clF^{-1}_{\dsR}\left[\frac{e^{-(2\beta-\pi)(\cdot-\frac{j}{2})}e^{-\pi(\cdot)}-e^{-(\beta-\pih+s)(\cdot-\frac{j}{2})}e^{-(\pih+t)(\cdot)}}{4\ch[\pi(\cdot)]\ch[(2\beta-\pi)(\cdot-\frac{j}{2})]}\clF_{\dsR}\varphi_1(\cdot,\hat{j}\,)\right](x)\\ \nonumber
  &=\sum_{j\in\dsZ}e^{2\pi ij\gamma}\clF^{-1}_{\dsR}\left[m^{I}_{t,s}(\cdot,j)\clF_{\dsR}\varphi_1(\cdot,\hat{j}\,)\right](x)+\sum_{j\in\dsZ}e^{2\pi ij\gamma}\clF^{-1}_{\dsR}\left[m^{\II}_{t,s}(\cdot,j)\clF_{\dsR}\varphi_1(\cdot,\hat{j}\,)\right](x)\\ \label{TSF1}
  &=T^{I}_{t,s}\varphi(x,\gamma)+T^{\II}_{t,s}\varphi(x,\gamma),
\end{align}
where
\begin{align*}
  &m_{t,s}^{I}(\xi,j)=\frac{1}{8}\bigg[\frac{e^{-\pi\xi}-e^{-(\pih+t)\xi}}{\ch[\pi\xi]}\bigg]\!\bigg[\frac{e^{-(2\beta-\pi)(\xi-\frac{j}{2})}+e^{-(\beta-\pih+s)(\xi-\frac{j}{2})}}{\ch[(2\beta-\pi)(\xi-\frac{j}{2})]}\bigg];\\
  &m_{t,s}^{\II}\!(\xi,j)\!=\!\frac{1}{8}\bigg[\frac{e^{-\pi\xi}+e^{-(\pih+t)\xi}}{\ch[\pi\xi]}\bigg]\!\bigg[\frac{e^{-(2\beta-\pi)(\xi-\frac{j}{2})}-e^{-(\beta-\pih+s)(\xi-\frac{j}{2})}}{\ch[(2\beta-\pi)(\xi-\frac{j}{2})]}\bigg].
\end{align*}
Thus, similarly to the proof of Proposition \ref{TransferredLpBounds}, the operators $T^{I}_{t,s}$ and $T^{II}_{t,s}$ can be seen as a composition of simpler operators. Namely,
\begin{align}\label{CompositionSF1}
  &T^{I}_{t,s}\varphi(x,\gamma)=[\Lambda^I_s\circ\Xi^I_t]\varphi(x,\gamma)\\ \label{CompositionSF1II}
  &T^{\II}_{t,s}\varphi(x,\gamma)=[\Lambda^{\II}_s\circ\Xi^{\II}_t]\varphi(x,\gamma),
\end{align}
where $\Lambda^I_s$, $\Xi^{I}_t$, $\Lambda_s^{\II}$ and $\Xi^{\II}_t$ are defined by 
\begin{align*}
  &\Lambda^I_s\varphi(x,\gamma):=\sum_{j\in\dsZ}\frac{e^{2\pi ij\gamma}}{2\pi}\int_{\dsR}\frac{e^{-(2\beta-\pi)(\xi-\frac{j}{2})}+e^{-(\beta-\pih+s)(\xi-\frac{j}{2})}}{\ch[(2\beta-\pi)(\xi-\frac{j}{2})]}\clF_\dsR \varphi_1(\xi,\hat{j}\,)e^{ix\xi}\ d\xi; \\
  &\Xi^I_t \varphi(x,\gamma):=\frac{1}{2\pi}\int_{\dsR}\frac{e^{-\pi\xi}-e^{-(\pih+t)\xi}}{\ch[\pi\xi]}\clF_\dsR \varphi_1(\xi,\gamma)e^{ix\xi}d \xi;\\
  &\Lambda^{\II}_s\varphi(x,\gamma):=\sum_{j\in\dsZ}\frac{e^{2\pi
      ij\gamma}}{2\pi}\int_{\dsR}\frac{e^{-(2\beta-\pi)(\xi-\frac{j}{2})}-e^{-(\beta-\pih+s)(\xi-\frac{j}{2})}}{\ch[(2\beta-\pi)(\xi-\frac{j}{2})]}\clF_\dsR 
  \varphi_1(\xi,\hat{j}\,)e^{ix\xi}\ d\xi;\\ 
  &\Xi^{\II}_t \varphi(x,\gamma):=\frac{1}{2\pi}\int_{\dsR}\frac{e^{-\pi\xi}+e^{-(\pih+t)\xi}}{\ch[\pi\xi]}\clF_\dsR \varphi_1(\xi,\gamma)e^{ix\xi}d \xi.
\end{align*}

%

So, in order to prove \eqref{Limit0}, we study the operators $\Lambda_s, \Xi_t, \Lambda'_s$ and $\Xi'_t$ separately. The realization of $T^{I}_{t,s}$ and $T^{II}_{t,s}$ as composition of these operators is particularly effective since the parameters $t$ and $s$ become, in some sense, independent.  
\begin{prop}\label{LambdaS1}
The operator $\Lambda^{I}_s$ extends to a bounded linear operator
$$
\Lambda^I_s : L^p(\dsR\times\dsT)\to L^p(\dsR\times\dsT)
$$
for every $p\in(1,\infty)$. Moreover,  if $\opnorm{\Lambda^I_s}{p}$ denotes the operator norm of $\Lambda^I_s$, it holds
\begin{equation}\label{UniformBoundLambdaS1}
\sup_{s\in[0,\beta-\pih)}\opnorm{\Lambda^I_s}{p}<\infty.
\end{equation}
\end{prop}
\begin{proof}
By density it suffices to prove the theorem for a function $g$ of the form $g(x,\gamma)=\sum\limits_{j=-N}^N g(x,j) e^{2\pi ij\gamma} $ and each $g(\cdot,j)$ is in $C^\infty_0(\dsR)$. Then, similarly to the proof of Proposition \ref{TransferredLpBounds} for the operator $\lambda_s$, we obtain
\begin{align*}
  &\int_\dsR\int_0^1\big|\Lambda^I_s g(x,\gamma)\big|^p\ d\gamma dx\\
  &=\!\int_\dsR\!\int_0^1\!\!\bigg|\frac{1}{2\pi}\!\int_\dsR\frac{e^{-(2\beta-\pi)\xi}+e^{-(\beta-\pih+s)\xi}}{\ch[(2\beta-\pi)\xi]}\!\clF_\dsR\!\bigg[\!\sum_{j=-N}^N e^{i\frac{j}{2}(\cdot)}(\cdot)g(\cdot,j)e^{2\pi ij\gamma}\bigg](\xi)e^{ix\xi}d\xi\bigg|^p\!d\gamma dx.
\end{align*}
By Mihlin's condition, we obtain that the function 
\begin{equation}\label{BoundsNorms}
\xi\mapsto\frac{e^{-(2\beta-\pi)\xi}+e^{-(\beta-\pih+s)\xi}}{\ch[(2\beta-\pi)\xi]}
\end{equation}
identifies a multiplier operator that is bounded on $L^p(\dsR)$ for every $p\in(1,\infty)$ and that satisfies (\ref{UniformBoundLambdaS1}). Notice also that the function $\sum\limits_{j=-N}^N e^{i\frac{j}{2}x}g(x,j)e^{2\pi ij\gamma} $ is in $L^p(\dsR\times\dsT)$.

Finally, by Fubini's theorem,
\begin{align*}
  &\int_\dsR\int_0^1\big|\Lambda^I_s g(x,\gamma)\big|^p\  d\gamma dx\\
  &=\int_0^1 \int_\dsR \bigg|\frac{1}{2\pi}\int_\dsR\frac{e^{-(2\beta-\pi)\xi}+e^{-(\beta-\pih+s)\xi}}{\ch[(2\beta-\pi)\xi]}\clF_\dsR\bigg[\sum_{j=-N}^N e^{i\frac{j}{2}}(\cdot)g(\cdot,j)e^{2\pi ij\gamma}\bigg](\xi)e^{ix\xi}d\xi\bigg|^p dxd\gamma \\
  &\leq C_p \int_0^1\int_\dsR \bigg|\sum_{j=-N}^N e^{i\frac{j}{2}x}g(x,j)e^{2\pi ij\gamma}\bigg|^p\ dxd\gamma\\
  &= C_p \int_0^1\int_\dsR \bigg|\sum_{j=-N}^N g(x,j)e^{2\pi ij\gamma}\bigg|^p\ dxd\gamma.
\end{align*}
\end{proof}
A similar argument proves the analogous result for the operator $\Xi^{\II}_t$.
\begin{prop}\label{LambdaT2}
The operator $\Xi^{\II}_t$ extends to a bounded linear operator
$$
\Xi^{\II}_t : L^p(\dsR\times\dsT)\to L^p(\dsR\times\dsT)
$$
for every $p\in(1,\infty)$. Moreover, if $\opnorm{\Xi^{\II}_t}{p}$ denotes the operator norm of $\Xi^{\II}_t$, it holds
\begin{equation}\label{UniformBoundLambdaT2}
\sup_{t\in[0,\pih)}\opnorm{\Xi^{\II}_t}{p}<\infty.
\end{equation}
\end{prop}
At last, we prove analogous proposition for the operators $\Xi^I_t$ and $\Lambda^{\II}_s$, but with an additional conclusion.
\begin{prop}\label{LambdaT1}
The operator $\Xi^I_t$ extends to a bounded linear operator 
$$
\Xi^I_t:L^p(\dsR\times\dsT)\to L^p(\dsR\times\dsT)
$$
for every $p\in(1,\infty)$. Moreover,
$$
\sup_{t\in[0,\pih)}\opnorm{\Xi^I_t}{p}<\infty
$$
and
$$
\lim_{t\to\pih}\|\Xi^I_t g\|_{L^p(\dsR\times\dsT)}=0
$$
for every function $g$ in $L^p(\dsR\times\dsT)$.
\end{prop}
\begin{proof}
The boundedness of $\Xi^I_t$ follows once again by Mihlin's condition, while the limit is computed as in Theorem \ref{SzegoStriscia} for the strip $S_\pih$. 
\end{proof}
\begin{prop}\label{LambdaS2}
The operator $\Lambda^{\II}_s$ extends to a bounded linear operator 
$$
\Lambda^{\II}_s:L^p(\dsR\times\dsT)\to L^p(\dsR\times\dsT)
$$
for every $p\in(1,\infty)$. Moreover,
$$
\sup_{s\in[0,\beta-\pih}\opnorm{\Lambda^{\II}_s}{p}<\infty.
$$
and
$$
\lim_{s\to\beta-\pih}\|\Lambda^{\II}_s g\|_{L^p(\dsR\times\dsT)}=0
$$
for every function $g$ in $L^p(\dsR\times\dsT)$.
\end{prop}
\begin{proof}
The proof follows similarly as the proofs of Proposition \ref{LambdaS1} and Proposition \ref{LambdaT1}
\end{proof}

Finally, using Propositions \ref{LambdaS1},\ref{LambdaT2}, \ref{LambdaT1} and \ref{LambdaS2} we obtain the proof of Theorem \ref{LimitSF1} for a function $\varphi$ of the form $\varphi=(\varphi_1, 0, 0, 0)$. The proof for a general function $\varphi$ follows with similar arguments. Moreover, Theorem \ref{HolomorphicityHp} and Theorem \ref{LimitSF1} prove Theorem \ref{t:Lp}.
\subsection{Sobolev regularity}
We can now prove Theorem \ref{t:Sobolev}.
\begin{proof}
 As usual, it suffices to prove the theorem for a function $\varphi=(\varphi_1, 0, 0, 0)$. For such a function $\varphi$, it holds
$$
\widetilde{S\varphi}_1(x+i\beta, e^{\frac{1}{2}(\beta-\pih)}e^{2\pi i\gamma})=\frac{1}{4}\sum_{j\in\dsZ}e^{2\pi ij\gamma}\clF_\dsR^{-1}\left[\frac{e^{-(2\beta-\pi)(\cdot-\frac{j}{2})}e^{-\pi(\cdot)}\clF_{\dsR}\varphi_1(\cdot,\hat{j}\,)}{\ch[\pi\cdot]\ch[(2\beta-\pi)(\cdot-\frac{j}{2})]}\right](x).
$$
We only prove explicitly that $\|\widetilde{S  \varphi}_1\|_{W^{k,p}(\dsR\times\dsT)}\leq \| \varphi\|_{W^{k,p}(\p D'_\beta)}$. With similar arguments, it is then possible to prove $\|\widetilde{S\varphi}\|_{W^{k,p}(\p D'_\beta)}\leq \|\varphi\|_{W^{k,p}(\p D'_\beta)}$.

Notice that
\begin{align*}
\sum_{j\in\dsZ}e^{2\pi ij\gamma}\clF^{-1}_{\dsR}&\Big[[1+j^2+(\cdot)^2]^{\frac{k}{2}}\clF_{\dsR}\widetilde{S\varphi}_1(\cdot,\widehat{j}\,)\Big](x)\\
&= \sum_{j\in\dsZ}e^{2\pi ij\gamma}\clF^{-1}_{\dsR}\Big[\frac{[1+j^2+(\cdot)^2]^\frac{k}{2}e^{-(2\beta-\pi)(\cdot-\frac{j}{2})}e^{-\pi(\cdot)}}{\ch[\pi\cdot]\ch[(2\beta-\pi)(\cdot-\frac{j}{2})]}\clF_\dsR\varphi_1(\cdot,\widehat{j})\Big](x)\\
&=\widetilde{S\psi^k}_1(x,\gamma),
\end{align*}
where $\psi^k=(\psi^k_1, 0, 0, 0)$ with
$$
\psi^k_1(x,\gamma):=\sum_{j\in\dsZ}e^{2\pi ij\gamma}\clF^{-1}_{\dsR}\Big[[1+j^2+(\cdot)^2]^{\frac{k}{2}}\clF_\dsR\varphi_1(\cdot,\widehat{j})\Big](x).
$$
Thus,
\begin{align*}
  \|\widetilde{S\varphi}_1\|^p_{W^{k,p}(E_1)}=\|\widetilde{S\psi^k}_1\|_{L^p(E_1)}^p
\end{align*}
and the conclusion follows from the $L^p$ boundedness of the operator $\widetilde{S}$.
\end{proof}
\subsection{A decomposition of $H^{\lowercase{p}}(D'_\beta)$}
In this section we prove that the the space $H^p(D'_\beta)$ admits for every $p\in(1,\infty)$ a decomposition 
\begin{equation}\label{DirectSumHp}
H^p(D'_\beta)=\bigoplus_{j\in\dsZ}\clH^p_j
\end{equation}
analogously to (\ref{DirectSumH2}) for the case $p=2$. Recall that, for every $j\in\dsZ$,
$$
\clH^p_j=\left\{F\in H^p(D'_\beta): F(z_1, e^{2\pi i\theta z_2})= e^{2\pi ij\theta}F(z_1,z_2)\right\}.
$$
Thus, we will prove that given a function $F$ in $H^p(D'_\beta)$, there exist functions $F_j$'s such that
$$
\lim_{N\to\infty}\bigg|\!\bigg|F-\sum_{j=-N}^N F_j\bigg|\!\bigg|_{H^p(D'_\beta)}=0,
$$
where each function $F_j$ belongs to $\clH^p_j$.

At first, we prove the result for functions which belong to the range of the operator $S$. Once again, without losing generality, we prove the result using simplified initial data and the general result will follow by linearity. Given a function $\varphi=(\varphi_1, 0, 0, 0)$ in $L^p(\partial D'_\beta)$, we define  
\begin{align*}
S_N\varphi(x+iy,e^\frac{s}{2}e^{2\pi i\gamma})&:=\sum_{j=-N}^N e^{2\pi ij\gamma}\clF^{-1}_{\dsR}\left[\frac{e^{-(\beta-\pih+s)(\cdot-\frac{j}{2})}e^{-(\pih-s+y)(\cdot)}\clF_{\dsR}\varphi_1(\cdot,\hat{j}\,)}{4\ch[\pi\cdot]\ch[(2\beta-\pi)(\cdot-\frac{j}{2})]}\right](x)\\
&=\sum_{j=-N}^N S_j\varphi(x+iy,e^{\frac{s}{2}}e^{2\pi i j\gamma}).
\end{align*}
Notice that each function $S_j\varphi$ trivially belongs to $\clH^p_j$.
\begin{prop}\label{SumNormConvergence}
Let $\varphi=(\varphi_1, 0, 0, 0)$ be a function in $L^p(D'_\beta)$, $p\in(1,\infty)$. Then, 
$$
\lim_{N\to\infty}\|S\varphi-S_N\varphi\|_{H^p(D'_\beta)}=0.
$$
\end{prop}
\begin{proof}
For almost every function $x\in\dsR$, the function $\varphi_1(x,\cdot)$ is in $L^p(\dsR)$. Thus, the $L^p$ convergence of one-dimensional Fourier series guarantees that
$$
\lim_{N\to\infty}\int_0^1 |\varphi_1(x,y)-\varphi_1^{(N)}(x,\gamma)|^p\ d\gamma=0
$$  
where $\varphi_1^{(N)}(x,\gamma)=\sum_{j=-N}^N \varphi_1(x,\hat{j}\,)e^{2\pi ij\gamma}$ and the limit holds almost everywhere. By means of Lebesgue's dominated convergence theorem, we can conclude that
\begin{align*}
  \lim_{N\to\infty}\int_{\dsR}\int_0^1 |\varphi_1(x,\gamma)-\varphi_1^{(N)}(x,\gamma)|^p\ d\gamma dx&=\int_{\dsR}\lim_{N\to\infty}\int_0^1 |\varphi_1(x,\gamma)-\varphi_1^{(N)}(x,\gamma)|^p\ d\gamma dx\\
  &=0.
\end{align*}
Thus we conclude that $\|\varphi-\varphi^{(N)}\|_{L^p}\to 0$ as $N$ tends to $+\infty$.
where $\varphi^{(N)}=(\varphi_1^{(N)}, 0, 0, 0)$.
By definition, see \eqref{SimplySFGeneral}, it holds
\begin{align*}
S[\varphi^{(N)}](x+iy,e^{\frac{s}{2}}e^{2\pi i\gamma})&=\sum_{j=-N}^{N}e^{2\pi ij\gamma}\clF^{-1}_{\dsR}\left[\frac{e^{-(\beta-\pih+s)(\cdot-\frac{j}{2})}e^{-(\pih-s+y)(\cdot)}\clF_{\dsR}\varphi_1(\cdot,\hat{j}\,)}{4\ch[\pi\cdot]\ch[(2\beta-\pi)(\cdot-\frac{j}{2})]}\right](x)\\
&=S_N\varphi(x+iy,e^{\frac{s}{2}}e^{2\pi i\gamma})\\
&=\sum_{j=-N}^N S_j\varphi(x+iy,e^{\frac{s}{2}}e^{2\pi i\gamma}).
\end{align*}
Finally, using estimate (\ref{SFHp}), we get
\begin{align*}
  \lim_{N\to\infty}\|S\varphi-S_N \varphi\|_{H^p(D'_\beta)}&=\lim_{N\to\infty}\|S\varphi-S[\varphi^{(N)}]\|_{H^p(D'_\beta)}\leq C_p \lim_{N\to\infty}\|\varphi-\varphi^{(N)}\|_{L^p(\partial D'_\beta)}= 0.
\end{align*}
The proof is complete.
\end{proof}
To obtain (\ref{DirectSumHp}) it remains to prove that the operator $S$ is surjective on $H^p(D'_\beta)$. We already know this is the case for $p=2$; the general case $p\in(1,\infty)$ will follow as a corollary of the following proposition. 
\begin{prop}\label{DensityHp}
For every $p$ in $(1,\infty)$, we have
$$
\overline{H^2(D'_\beta)\cap H^p(D'_\beta)}^{\norm{\cdot}{H^p}} = H^p(D'_\beta).
$$
\end{prop}
\begin{proof}
For every $\varepsilon>0$ and $z_1\in \mathcal{S}_\beta$ consider the function
$$
G^{\varepsilon}(z_1)=\frac{1}{1+\varepsilon[2\beta+iz_1]}.
$$
It can be proved that, for every fixed $\varepsilon>0$ and $F\in H^p(D'_\beta)$, the function $F\cdot G^\varepsilon$ is in $H^2(D'_\beta)\cap H^p(D'_\beta)$, thus $FG^\varepsilon=S[\widetilde{FG^\varepsilon}]$. Notice that $G^\varepsilon$ admits a continuous extension to $\overline{D'_\beta}$, therefore $\widetilde{FG^\varepsilon}=\widetilde{F}G^\varepsilon$, where $\widetilde{F}$ is the weak-$\ast$ limit of $F$ (see Proposition \ref{Weak*Hp}). Now,
\begin{align*}
  \lim_{\varepsilon\to 0^+}\|F-FG^\varepsilon\|^p_{H^p(D'_\beta)}&\leq\lim_{\varepsilon\to 0^+}\sup_{(t,s)}\int_0^1\int_\dsR|(F-FG^\varepsilon)[x\pm i(s+t), e^{\pm\frac{s}{2}}e^{2\pi i\gamma}]|^p\ dxd\gamma\\
  &\ \ \ +\lim_{\varepsilon\to 0^+}\sup_{(t,s)}\int_0^1\int_\dsR|(F-FG^\varepsilon)[x\pm i(s-t), e^{\pm\frac{s}{2}}e^{2\pi i\gamma}]|^p\ dxd\gamma.
\end{align*}
We focus on one of these term; the computation for the other terms is similar. Therefore, using Fatou's lemma,
\begin{align*}
  &\lim_{\varepsilon\to 0^+}\sup_{(t,s)}\int_0^1\int_\dsR |(F-FG^\varepsilon)[x+i(s+t), e^{\frac{s}{2}}e^{2\pi i\gamma}]|^p\ dxd\gamma\\
  &\ \ =\lim_{\varepsilon\to 0^+}\sup_{(t,s)}\int_0^1\int_\dsR\big|F[x+i(s+t),e^{\frac{s}{2}}e^{2\pi i\gamma}]\big[1-G^{\varepsilon}[x+i(s+t)]\big]\big|^p\ dxd\gamma\\
  &\ \ \leq\lim_{\varepsilon\to 0^+}\sup_{(t,s)}\liminf_{\delta\to 0^+}\int_0^1\int_\dsR\big|F[x+i(s+t),e^{\frac{s}{2}}e^{2\pi i\gamma}]\big[(G^\delta-G^{\varepsilon})[x+i(s+t)]\big]\big|^p\ dxd\gamma\\
  &\ \ =\lim_{\varepsilon\to 0^+}\sup_{(t,s)}\liminf_{\delta\to 0^+}\int_0^1\int_\dsR \big|S[\widetilde{F}(G^\delta-G^\varepsilon)][x+i(s+t),e^{\frac{s}{2}}e^{2\pi i\gamma}]\big|^p\ dxd\gamma\\
  &\ \ \leq \lim_{\varepsilon\to 0^+}\sup_{(t,s)}\liminf_{\delta\to 0^+}\|S[\widetilde{F}(G^\delta-G^\varepsilon)]\|^p_{H^p(D'_\beta)}\\
  &\ \ \leq C_p \lim_{\varepsilon\to 0^+}\liminf_{\delta\to 0^+}\|\widetilde{F}(G^\delta-G^\varepsilon)\|_{L^p(\partial D'_\beta)}\\
  &\ \ =0,
\end{align*}
where in the last two lines we used the boundedness of the operator $S$ and the dominated convergence theorem. The proof is complete.
\end{proof}
\begin{cor}\label{SurjectiveS}
Let $F$ be a function in $H^p(D'_\beta)$, $p\in(1,\infty)$. Then, there exists a function $\varphi$ in $L^p(\partial D'_\beta)$ such that $F=S\varphi$.
\end{cor}

\begin{rmk}
Theorem \ref{LimitSF1} shows that every function in the range of $S$ tends to its boundary values in norm. The previous corollary allows to conclude that this is true for every element of $H^p(D'_\beta)$, $p\in(1,\infty)$.
\end{rmk}
\begin{rmk}
Proposition \ref{SumNormConvergence} and Corollary \ref{SurjectiveS} together prove the decomposition (\ref{DirectSumHp}).
\end{rmk}

\subsection{Pointwise convergence}\label{PointConv}
We conclude the section showing that
an appropriate restriction of a function $F$ in $H^p(D'_\beta)$,
$p\in(1,\infty)$, converges to its boundary value function also pointwise almost everywhere . As usual, we prove
our results in a simplified situation and the general case follows by
linearity. Let $\varphi=(\varphi_1, 0, 0, 0)$ be a function in $L^p(\partial
D'_\beta)$, then we proved that
$$
\lim_{(t,s)\to (\pih,\beta-\pih)}\int_{\dsR}\int_0^1 
\Big|S\varphi[x+i(s+t),e^{\frac{s}{2}}e^{2\pi i\gamma}]-\widetilde{S\varphi}_1[x+i\beta, e^{\frac{1}{2}(\beta-\pih)}e^{2\pi i\gamma}]\Big|^p\ d\gamma dx=0.
$$
In general, to prove a pointwise convergence result, we expect we
need to put some restrictions on the parameters $t$ and $s$. For
example, also in the simpler case of the polydisc
$D^2(0,1)=D(0,1)\times D(0,1)$, we are able to prove the almost
everywhere existence of the pointwise radial limit 
$$
\lim_{(r_1, r_2)\to (1,1)} G(r_1 e^{2\pi i\theta},r_2 e^{2\pi i\gamma})
$$ 
for a function $G$ in $H^p(D^2)$ under the hypothesis that the ratio
$\frac{1-r_1}{1-r_2}$ is bounded (see, for example, \cite[Chapter 2,
Section 2.3]{MR0255841}).  

At the moment, we are able to prove a pointwise convergence result
which depends only on one parameter. It would be interesting to
determine a larger approach region to the distinguished boundary
$\partial D'_\beta$ and discuss non-tangential convergence.

We need the following lemma which it is not hard to prove using the results contained in Section \ref{strip}.
\begin{lem}\label{l:strip}
Let $S_\beta$ be the strip $S_\beta=\{z=x+iy\in\dsC:|y|<\beta\}$. Let $\varphi=(\varphi_+,\varphi_-)$ be a function in $L^p(\partial S_\beta)$, $p\in(1,\infty)$. Then, the function 
$$
S_{S_\beta}\varphi(x+iy)=\clF^{-1}\left[\frac{e^{-(y+\beta)(\cdot)}\widehat{\varphi_+}+e^{-(y-\beta)(\cdot)}\widehat{\varphi_-}}{4\ch[\pi(\cdot)]\ch[(2\beta-\pi)(\cdot-\frac{j}{2})]}\right](x)
$$
belongs to $H^p(S_\beta)$ for every integer $j$.
\end{lem}

\begin{thm}\label{SimplyPC}
Let $\varphi=(\varphi_1, \varphi_2, \varphi_3, \varphi_4)$ be a function in $L^p(\partial D'_\beta)$, $p\in(1,\infty)$. Then,
\begin{equation}\label{LimitPSF1}
\lim_{t\to\beta^-}S\varphi[x+it,e^{\frac{t}{2\beta}(\beta-\pih)}e^{2\pi i\gamma}]=\widetilde{S\varphi}_1[x+i\beta, e^{\frac{1}{2}(\beta-\pih)}e^{2\pi i\gamma}]
\end{equation}
for almost every $(x,\gamma)\in \dsR\times\dsT$.
\end{thm}
\begin{proof}
We prove the theorem for $\varphi=(\varphi_1, 0, 0, 0)$.
By (\ref{TSF1}), we want to prove that
\begin{align*}
L_t(x,\gamma)&=\bigg|\sum_{j\in\dsZ}e^{2\pi ij\gamma}\clF^{-1}_{\dsR}\bigg[\frac{e^{-2\beta(\cdot)}e^{j(\beta-\pih)}-e^{-(\beta+t)(\cdot)}e^{\frac{j}{2}(\beta-\pih)(1+\frac{t}{\beta})}}{4\ch[\pi(\cdot)]\ch[(2\beta-\pih)(\cdot-\frac{j}{2})]}\clF_{\dsR}\varphi_1(\cdot,\hat{j}\,)\bigg](x)\bigg|\\
&=:\bigg|\sum_{j\in\dsZ}S_j^t \varphi(x,\gamma)\bigg|\to0
\end{align*}
for almost every $(x,\gamma)$ $\in\dsR\times\dsT$ when $t$ tends to $\beta^-$. Let $\varepsilon>0$ be fixed. Then,
\begin{align*}
  \bigg|\bigg\{(x,\gamma): \limsup_{t\to\beta^-}&\, L_t(x,\gamma)>\varepsilon \bigg\}\bigg|\leq\sum_{j\in\dsZ}\bigg|\bigg\{(x,\gamma): \limsup_{t\to\beta^-}|S_j^t \varphi(x,\gamma)|>\alpha_j\bigg\}\bigg|,
\end{align*}
where the $\alpha_j$'s are positive and $\sum_{j\in\dsZ}\alpha_j=\varepsilon$. We claim that the sets in the right-hand side of the previous inequality are all of measure zero. Following the proof of Theorem \ref{LimitSF1} we obtain that
\begin{equation}\label{NormLimitJ}
\lim_{t\to\beta^-}\norm{S^t_j\varphi}{L^p(\dsR\times\dsT)}=0.
\end{equation}
Therefore, it is enough to prove the existence of the pointwise limit \begin{align*}
\lim_{t\to\beta^-}e^{2\pi i j\gamma}\clF^{-1}_{\dsR}
&\bigg[\frac{e^{-(\beta+t)(\cdot)}e^{\frac{j}{2}(\beta-\pih)(1+\frac{t}{\beta})}}{4\ch[\pi(\cdot)]\ch[(2\beta-\pi)(\cdot-\frac{j}{2})]}\clF_\dsR \varphi_1(\cdot,\hat{j}\,)\bigg](x)
\end{align*}
for almost every $(x,\gamma)\in\dsR\times\dsT$. To prove this, it is sufficient to prove that
$$
\lim_{t\to\beta^-}\clF^{-1}_{\dsR}\bigg[\frac{e^{-(\beta+t)(\cdot)}\clF_{\dsR} G(\cdot)}{4\ch[\pi(\cdot)]\ch[(2\beta-\pi)(\cdot-\frac{j}{2}]}\bigg](x)
$$
exists for almost every $x$ in $\dsR$ and for every function $G$ in $L^p(\dsR)$, $p\in(1,\infty)$. The existence of this last limit follows immediately from the lemma and Theorem \ref{SzegoStriscia}.

Analogously it can be proved the pointwise convergence of $S\varphi$ to the other components of $\partial D'_\beta$.
\end{proof} 

\begin{rmk}
We proved the previous theorem for functions that belong to the range of the operator $S$. From Proposition \ref{SurjectiveS} we can conclude that the result is true for every function in $H^p(D'_\beta), p\in(1,\infty)$.
\end{rmk}

\subsection{A density result}
At last, we use the $L^p$ boundedness of the operator $\widetilde S$ to prove that continuous functions on the closure of $D'_\beta$ are dense in $H^p(D'_\beta)$.
\begin{thm}\label{t:density}
Let $p\in(1,\infty)$. Then,
$$
\overline{H^p(D'_\beta)\cap\clC(\overline{D'_\beta})}^{\|\cdot\|_{H^p(D'_\beta)}}=H^p(D'_\beta).
$$
\end{thm}
\begin{proof}
It is enough to prove that for a given function $\varphi=(\varphi_1, 0, 0 ,0)$ where $\varphi_1(x,\gamma)=\sum_{j=-N}^N\varphi_1(x,j)e^{2\pi ij\gamma}$ with $\varphi_{1}(\cdot,j)\in \clC^{\infty}_0(\dsR)$ for every $j$, then $S\varphi$ belongs to $H^p(D'_\beta)\cap\clC(\overline{D'_\beta})$. The conclusion will follow by linearity, density and the $L^p$ boundedness of the Szeg\H{o} projector $\widetilde{S}$. We have
$$
S\varphi(x+iy, e^{\frac{s}{2}}e^{2\pi i\gamma})=\sum_{j=-N}^{N}e^{2\pi ij\gamma}\clF^{-1}_{\dsR}\left[\frac{e^{-(\beta-\pih+s)(\cdot-\frac{j}{2})}e^{-(\pih-s+y)(\cdot)}\clF_{\dsR}\varphi_1(\cdot,j)}{4\ch[\pi\cdot]\ch[(2\beta-\pi)(\cdot-\frac{j}{2})]}\right](x)
$$
and $S\varphi$ is continuous up to the boundary of $\p D'_\beta$ since each term of the sum is thanks to Lemma \ref{l:strip} and Remark \ref{ContinuityBoundary}. The proof is complete.
\end{proof}

\bibliographystyle{plain}
\bibliography{bibWormStrip}

\begin{thebibliography}{10}

\bibitem{MR2323495}
A.~Bakan and S.~Kaijser.
\newblock Hardy spaces for the strip.
\newblock {\em J. Math. Anal. Appl.}, 333(1):347--364, 2007.

\bibitem{MR1149863}
D.~E. Barrett.
\newblock Behavior of the {B}ergman projection on the {D}iederich-{F}orn\ae ss
  worm.
\newblock {\em Acta Math.}, 168(1-2):1--10, 1992.

\bibitem{2014arXiv1408.0082B}
D.~E. {Barrett}, D.~{Ehsani}, and M.~M. {Peloso}.
\newblock {Regularity of projection operators attached to worm domains}.
\newblock {\em ArXiv e-prints}, August 2014.

\bibitem{MR2904008}
D.~E. Barrett and S.~{\c{S}}ahuto{\u{g}}lu.
\newblock Irregularity of the {B}ergman projection on worm domains in {$\Bbb
  C^n$}.
\newblock {\em Michigan Math. J.}, 61(1):187--198, 2012.

\bibitem{MR3145917}
David Barrett and Lina Lee.
\newblock On the {S}zeg{\H o} metric.
\newblock {\em J. Geom. Anal.}, 24(1):104--117, 2014.

\bibitem{MR568937}
S.~Bell and E.~Ligocka.
\newblock A simplification and extension of {F}efferman's theorem on
  biholomorphic mappings.
\newblock {\em Invent. Math.}, 57(3):283--289, 1980.

\bibitem{MR625347}
S.~R. Bell.
\newblock Biholomorphic mappings and the {$\bar \partial $}-problem.
\newblock {\em Ann. of Math. (2)}, 114(1):103--113, 1981.

\bibitem{MR1133741}
H.~P. Boas and E.~J. Straube.
\newblock Sobolev estimates for the complex {G}reen operator on a class of
  weakly pseudoconvex boundaries.
\newblock {\em Comm. Partial Differential Equations}, 16(10):1573--1582, 1991.

\bibitem{MR773403}
Harold~P. Boas.
\newblock Regularity of the {S}zeg{\H o} projection in weakly pseudoconvex
  domains.
\newblock {\em Indiana Univ. Math. J.}, 34(1):217--223, 1985.

\bibitem{MR871667}
Harold~P. Boas.
\newblock The {S}zeg{\H o} projection: {S}obolev estimates in regular domains.
\newblock {\em Trans. Amer. Math. Soc.}, 300(1):109--132, 1987.

\bibitem{MR971689}
Harold~P. Boas, So-Chin Chen, and Emil~J. Straube.
\newblock Exact regularity of the {B}ergman and {S}zeg{\H o} projections on
  domains with partially transverse symmetries.
\newblock {\em Manuscripta Math.}, 62(4):467--475, 1988.

\bibitem{MR999739}
Harold~P. Boas and Emil~J. Straube.
\newblock Complete {H}artogs domains in {$\mathbf C^2$} have regular {B}ergman
  and {S}zeg{\H o} projections.
\newblock {\em Math. Z.}, 201(3):441--454, 1989.

\bibitem{MR1094488}
So-Chin Chen.
\newblock Real analytic regularity of the {S}zeg{\H o} projection on circular
  domains.
\newblock {\em Pacific J. Math.}, 148(2):225--235, 1991.

\bibitem{MR1370592}
M.~Christ.
\newblock Global {$C^\infty$} irregularity of the
  {$\overline\partial$}-{N}eumann problem for worm domains.
\newblock {\em J. Amer. Math. Soc.}, 9(4):1171--1185, 1996.

\bibitem{MR1381988}
M.~Christ.
\newblock The {S}zeg{\H o} projection need not preserve global analyticity.
\newblock {\em Ann. of Math. (2)}, 143(2):301--330, 1996.

\bibitem{MR906810}
Katharine~Perkins Diaz.
\newblock The {S}zeg{\H o} kernel as a singular integral kernel on a family of
  weakly pseudoconvex domains.
\newblock {\em Trans. Amer. Math. Soc.}, 304(1):141--170, 1987.

\bibitem{MR0430315}
K.~Diederich and J.~E. Fornaess.
\newblock Pseudoconvex domains: an example with nontrivial {N}ebenh\"ulle.
\newblock {\em Math. Ann.}, 225(3):275--292, 1977.

\bibitem{MR2445437}
L.~Grafakos.
\newblock {\em Classical {F}ourier analysis}, volume 249 of {\em Graduate Texts
  in Mathematics}.
\newblock Springer, New York, second edition, 2008.

\bibitem{MR3272760}
P.~S. Harrington, M.~M. Peloso, and A.~S. Raich.
\newblock Regularity equivalence of the {S}zeg\"o projection and the complex
  {G}reen operator.
\newblock {\em Proc. Amer. Math. Soc.}, 143(1):353--367, 2015.

\bibitem{MR1128596}
C.~O. Kiselman.
\newblock A study of the {B}ergman projection in certain {H}artogs domains.
\newblock In {\em Several complex variables and complex geometry, {P}art 3
  ({S}anta {C}ruz, {CA}, 1989)}, volume~52 of {\em Proc. Sympos. Pure Math.},
  pages 219--231. Amer. Math. Soc., Providence, RI, 1991.

\bibitem{MR2393268}
S.~G. Krantz and M.~M. Peloso.
\newblock Analysis and geometry on worm domains.
\newblock {\em J. Geom. Anal.}, 18(2):478--510, 2008.

\bibitem{2014arXiv1410.8490K}
S.~G. {Krantz}, M.~M. {Peloso}, and C.~{Stoppato}.
\newblock {Bergman kernel and projection on the unbounded worm domain}.
\newblock {\em ArXiv e-prints}, October 2014.

\bibitem{MR2336324}
S.G. Krantz and M.~M. Peloso.
\newblock New results on the {B}ergman kernel of the worm domain in complex
  space.
\newblock {\em Electron. Res. Announc. Math. Sci.}, 14:35--41 (electronic),
  2007.

\bibitem{MR2448387}
S.G. Krantz and M.~M. Peloso.
\newblock The {B}ergman kernel and projection on non-smooth worm domains.
\newblock {\em Houston J. Math.}, 34(3):873--950, 2008.

\bibitem{MR3084008}
L.~Lanzani and E.~M. Stein.
\newblock Cauchy-type integrals in several complex variables.
\newblock {\em Bull. Math. Sci.}, 3(2):241--285, 2013.

\bibitem{MR2030575}
Loredana Lanzani and Elias~M. Stein.
\newblock Szeg\"o and {B}ergman projections on non-smooth planar domains.
\newblock {\em J. Geom. Anal.}, 14(1):63--86, 2004.

\bibitem{MR1452048}
J.~D. McNeal and E.~M. Stein.
\newblock The {S}zeg{\H o} projection on convex domains.
\newblock {\em Math. Z.}, 224(4):519--553, 1997.

\bibitem{MonThesis}
A.~Monguzzi.
\newblock {\em {On the regularity of singular integrals operators on complex
  domains}}.
\newblock PhD thesis, Universit\`{a} degli Studi di Milano, 2015.

\bibitem{MR979602}
A.~Nagel, J.-P. Rosay, E.~M. Stein, and S.~Wainger.
\newblock Estimates for the {B}ergman and {S}zeg{\H o} kernels in {${\bf
  C}^2$}.
\newblock {\em Ann. of Math. (2)}, 129(1):113--149, 1989.

\bibitem{MR1451142}
R.~E. A.~C. Paley and N.~Wiener.
\newblock {\em Fourier transforms in the complex domain}, volume~19 of {\em
  American Mathematical Society Colloquium Publications}.
\newblock American Mathematical Society, Providence, RI, 1987.
\newblock Reprint of the 1934 original.

\bibitem{MR0450623}
D.~H. Phong and E.~M. Stein.
\newblock Estimates for the {B}ergman and {S}zeg\"o projections on strongly
  pseudo-convex domains.
\newblock {\em Duke Math. J.}, 44(3):695--704, 1977.

\bibitem{MR0255841}
Walter Rudin.
\newblock {\em Function theory in polydiscs}.
\newblock W. A. Benjamin, Inc., New York-Amsterdam, 1969.

\bibitem{MR0369703}
A.~M. Sedlecki{\u\i}.
\newblock An equivalent definition of the {$H^{p}$} spaces in the half-plane,
  and some applications.
\newblock {\em Mat. Sb. (N.S.)}, 96(138):75--82, 167, 1975.

\bibitem{MR0473215}
E.~M. Stein.
\newblock {\em Boundary behavior of holomorphic functions of several complex
  variables}.
\newblock Princeton University Press, Princeton, N.J.; University of Tokyo
  Press, Tokyo, 1972.
\newblock Mathematical Notes, No. 11.

\bibitem{MR835396}
Emil~J. Straube.
\newblock Exact regularity of {B}ergman, {S}zeg{\H o} and {S}obolev space
  projections in nonpseudoconvex domains.
\newblock {\em Math. Z.}, 192(1):117--128, 1986.

\end{thebibliography}

\end{document}